\documentclass[a4paper,12pt]{article}

\usepackage[utf8]{inputenc}
\usepackage[english]{babel}
\usepackage{hyperref}
\usepackage{amssymb,amsmath,amsthm}
\usepackage{enumerate}
\usepackage{enumitem}
\usepackage{amssymb}
\usepackage{graphicx, color}
\usepackage{epsfig}
\usepackage{tikz}
\usetikzlibrary{calc,decorations.pathmorphing,decorations.text}
\usepackage{subcaption}
\usepackage{refcount}
\usepackage[hmargin=2.5cm,vmargin=3cm]{geometry}
\usepackage{mathtools, braket}
\usepackage[normalem]{ulem}
\usepackage{authblk}
\usepackage{centernot}
\usepackage{indentfirst}
\usepackage{booktabs}

\usepackage{hyperref}
\hypersetup{
	colorlinks=true,
	linkcolor=blue,
	citecolor=gray}
\usepackage[square,sort,comma,numbers]{natbib}
\usepackage{caption}
\usetikzlibrary{matrix}
\usetikzlibrary{decorations.markings}
\tikzstyle{vertex}=[circle, draw, fill=black!50,
inner sep=0pt, minimum width=4pt]
\tikzset{->-/.style={decoration={
			markings,
			mark=at position .5 with {\arrow{>}}},postaction={decorate}}}
\tikzstyle{bigblue}=[color=blue, very thick, >=stealth]
\tikzstyle{lightblue}=[color=blue, thin, >=stealth]
\tikzstyle{bigred}=[color=red, very thick, >=stealth]
\tikzstyle{lightred}=[color=red, thin, >=stealth]
\tikzstyle{biggreen}=[color=black!30!green, very thick, >=stealth]
\tikzstyle{lightgreen}=[color=black!30!green,  thin, >=stealth]

\usepackage{ulem}
\usepackage[thicklines]{cancel}
\usepackage{xcolor}

\usepackage[color=green!30]{todonotes}

\usepackage{caption}
\usetikzlibrary{matrix}
\usetikzlibrary{decorations.markings}
\tikzstyle{vertex}=[circle, draw, fill=black!50,
inner sep=0pt, minimum width=4pt]
\tikzset{->-/.style={decoration={
			markings,
			mark=at position .5 with {\arrow{>}}},postaction={decorate}}}

\tikzstyle{bigblue}=[color=blue, very thick, >=stealth]
\tikzstyle{lightblue}=[color=blue, thin, >=stealth]

\tikzstyle{bigred}=[color=red, very thick, >=stealth]
\tikzstyle{lightred}=[color=red, thin, >=stealth]

\tikzstyle{biggreen}=[color=black!30!green, very thick, >=stealth]
\tikzstyle{lightgreen}=[color=black!30!green,  thin, >=stealth]

\usepackage{amscd, color}
\usepackage{amsmath,amsfonts,amssymb,amscd}
\usepackage{pgfplots}
\usepackage{indentfirst,graphicx,epsfig}
\usepackage{graphicx,psfrag}
\usepackage{tikz}
\usetikzlibrary{calc,decorations.pathmorphing,decorations.text}
\input{epsf}

\newtheorem{thm}{Theorem}

\newtheorem{lem}{Lemma}
\newtheorem{ob}{Observation}
\newtheorem{claim}{Claim}
\newtheorem{Constrution}{Construction}
\newtheorem{Pro}{Proposition}

\newtheorem{cor}{Corollary}

\title{\bf The degree and codegree threshold for generalized triangle and some trees covering}
\author{
	\small  Ran Gu$^1$,  %thanks{Corresponding author.}
	Shuaichao Wang$^2$\\
	\small $^1$College of Science, Hohai University,\\
	\small Nanjing, Jiangsu Province 210098,
	P.R. China\\
	\small $^2$Center for Combinatorics and LPMC \\
	\small Nankai University, Tianjin 300071, China \\
	\small Emails: rangu@hhu.edu.cn;wsc17746316863@163.com
	\\
	\date{}}

\makeatletter

\newcommand{\Rmnum}[1]{\expandafter@slowromancap\romannumeral #1@}
\makeatother
\begin{document}
	\maketitle
	\begin{abstract}
		Given two $k$-uniform hypergraphs $F$ and $G$, we say that $G$ has an $F$-covering if for every vertex in $G$ there is a  copy of $F$ covering it. For $1\leq i\leq k-1$, the minimum $i$-degree $\delta_i(G)$ of $G$ is the minimum integer such that every $i$ vertices are contained in at least $\delta_i(G)$ edges. Let $c_i(n,F)$ be the largest minimum $i$-degree among all $n$-vertex $k$-uniform hypergraphs that  have no $F$-covering. In this paper, we consider the $F$-covering problem in $3$-uniform hypergraphs when $F$ is the generalized triangle $T$, where $T$ is a $3$-uniform hypergraph with the vertex set $\{v_1,v_2,v_3,v_4,v_5\}$ and the edge set $\{\{v_{1}v_{2}v_{3}\},\{v_{1}v_{2}v_{4}\},\{v_{3}v_{4}v_{5}\}\}$. We give the exact value of $c_2(n,T)$ and asymptotically determine $c_1(n,T)$. We also consider the $F$-covering problem in $3$-uniform hypergraphs when $F$ are some trees, such as the linear $k$-path $P_k$ and the star $S_k$. Especially, we provide bounds of $c_i(n,P_k)$ and $c_i(n,S_k)$ for $k\geq 3$, where $i=1,2$.
		\\[2mm]

		\noindent\textbf{Keywords:}  $3$-graphs; Covering; Extremal;\\
		[2mm] \textbf{AMS Subject Classification (2020):} 05C35, 05C22
	\end{abstract}

	\section{Introduction}
	Let $k$ be  an integer with $k\geq 2$. We say a $k$-uniform hypergraph, or a $k$-graph, is a pair $G=(V,E)$, where $V$ is a set of vertices and $E$ is a collection of $k$-subsets of $V$. When $k=2$, the $k$-graph is the simple graph. We simply denote $2$-graph by graph. Let $G=(V,E)$ be a $k$-graph. For any $S\subset V(G)$, let the neighborhood $N_G(S)$ of $S$ be $\{T\subset V(G)\setminus S:T\cup S\in E(G)\}$ and the degree $d_G$ of $S$ be $|N_G(S)|$. For $1\leq i \leq k-1$, we denote the \emph{minimum $i$-degree} of $G$ by $\delta_i(G)$,  which is the minimum of $d_G(S)$ over all $S\in \binom{V(G)}{i}$. We call $\delta_1(G)$ the \emph{minimum degree} of $G$ and $\delta_{k-1}(G)$ the \emph{minimum codegree} of $G$. When $|S|=k-1$, we also call the vertex in $N_G(S)$ the co-neighbor of $S$. For a vertex $x$ in $V$, we define the link graph $G_x$ to be a $(k-1)$-graph with the vertex set $V(G)\setminus \{x\}$ and the edge set $N_G(\{x\})$.\par 
	Given a $k$-graph $F$, we say a $k$-graph $G$ has an $F$-covering if for any vertex of $G$, we can find a 
	copy of $F$ containing it. For $1\leq i \leq k-1$, define
	$$c_i(n,F)=\max\{\delta_i(G): \text{$G$ is a $k$-graph on $n$ vertices with no $F$-covering}\}.$$
	and call $c_1(n,F)$ the $F$-covering degree-threshold and $c_{k-1}(n,F)$ the $F$-covering codegree-threshold.\par 
	For graphs $F$, 
	Han, Zang and Zhao introduced the $F$-covering problem in \cite{ZM} and showed that $c_1(n,F)=(\frac{\chi(F)-2}{\chi(F)-1}+o(1))n$ where $\chi(F)$ is the chromatic number of $F$. Falgas-Ravry and Zhao \cite{FZC} initiated the study of the $F$-covering problem in 3-graphs. For $n\geq k$, let $K_n^k$ denote the complete $k$-graph on $n$ vertices and $K_n^{k-}$ denote the $k$-graph by removing one edge from $K_n^k$. In \cite{FZC}, Falgas-Ravry and Zhao determined the exact value of $c_2(n,K_4^3)$ for $n> 98$ and gave bounds of $c_2(n,F)$ when $F$ is $K_4^{3-}$, $K_5^{3}$ or the tight cycle $C_5^3$ on $5$ vertices. Yu, Hou, Ma and Liu in \cite{YHMLE} gave the exact value of $c_2(n,K_4^{3-})$, $c_2(n,K_5^{3-})$ and showed that $c_2(n,K_4^{3-})=\lfloor \frac{n}{3}\rfloor$, $c_2(n,K_5^{3-})=\lfloor \frac{2n-2}{3}\rfloor$. Soon after that, Falgas-Ravry, Markstr\"{o}m, and Zhao in  \cite{FMZT} gave near optimal bounds of $c_1(n,K_4^3)$ and asymptotically determined  $c_1(n,K_4^{3-})$. Recently, Tang, Ma and Hou in \cite{TMH} determined the exact value of $c_2(n,C_6^3)$ and an asymptotic optimal value of $c_1(n,C_6^3)$. There are some other related results in literature, for example in \cite{FTDR},\cite{FZS}.\par 
	In this paper, we also focus on  $3$-graphs.  Let the generalized triangle $T$ be a $3$-graph with the vertex set $\{v_1,v_2,v_3,v_4,v_5\}$ and the edge set $\{\{v_1v_2v_3\},\{v_1v_2v_4\},\{v_3\\v_4v_5\}\}$. We determine the exact value of $c_2(n,T)$ and give the bounds of $c_1(n,T)$. What's more, let $G$ be a graph and fix a vertex $u$ in $V(G)$. If $u$ is covered by a generalized triangle, then there are three possible positions for $u$ to have, see Figure \ref{fig:Tcover}. We denote these three ways by $T^1$, $T^2$ and $T^3$. We give the upper bound of $\delta_1(G)$ guaranteeing that every vertex in $V(G)$ is contained in $T^1$, $T^2$ and $T^3$. The main results on generalized triangle are as follows.

	\begin{figure}[!htbp]
		\centering
		\begin{subfigure}[!htbp]{.32\textwidth}
			\centering
			\begin{tikzpicture}[scale=0.6]
				
				\draw [rotate around={0:(5,0)},line width=1pt] (5,0) ellipse (3.3cm and 0.7cm);
				\draw [rotate around={37:(6,-1.5)},line width=1pt] (6,-1.5) ellipse (2.7cm and 0.7cm);
				\draw [rotate around={-37:(3.9,-1.4)},line width=1pt] (3.9,-1.4) ellipse (2.7cm and 0.7cm);
				
				\begin{scriptsize}
					\draw [fill=blue] (8,0)  circle (3.5pt);
					\draw [fill=blue] (2,0) circle (3.5pt);
					\draw [fill=red] (5,0) node [above] {\large $u$} circle (3.5pt);
					\draw [fill=blue] (4.6,-2.4) circle (3.5pt);
					\draw [fill=blue] (5.4,-2.4) circle (3.5pt);
				\end{scriptsize}
			\end{tikzpicture}
			\caption{$u$ is contained in $T^1$}
		\end{subfigure}	
		\begin{subfigure}[!htbp]{.32\textwidth}
			\centering
			\begin{tikzpicture}[scale=0.6]
				
				\draw [rotate around={0:(5,0)},line width=1pt] (5,0) ellipse (3.3cm and 0.7cm);
				\draw [rotate around={37:(6,-1.5)},line width=1pt] (6,-1.5) ellipse (2.7cm and 0.7cm);
				\draw [rotate around={-37:(3.9,-1.4)},line width=1pt] (3.9,-1.4) ellipse (2.7cm and 0.7cm);
				
				\begin{scriptsize}
					\draw [fill=red] (8,0)node [above] {\large $u$}  circle (3.5pt);
					\draw [fill=blue] (2,0) circle (3.5pt);
					\draw [fill=blue] (5,0)  circle (3.5pt);
					\draw [fill=blue] (4.6,-2.4) circle (3.5pt);
					\draw [fill=blue] (5.4,-2.4) circle (3.5pt);
				\end{scriptsize}
			\end{tikzpicture}
			\caption{$u$ is contained in $T^2$}
		\end{subfigure}	
		\begin{subfigure}[!htbp]{.32\textwidth}
			\centering
			\begin{tikzpicture}[scale=0.6]
				
				\draw [rotate around={0:(5,0)},line width=1pt] (5,0) ellipse (3.3cm and 0.7cm);
				\draw [rotate around={37:(6,-1.5)},line width=1pt] (6,-1.5) ellipse (2.7cm and 0.7cm);
				\draw [rotate around={-37:(3.9,-1.4)},line width=1pt] (3.9,-1.4) ellipse (2.7cm and 0.7cm);
				
				\begin{scriptsize}
					\draw [fill=blue] (8,0)  circle (3.5pt);
					\draw [fill=blue] (2,0) circle (3.5pt);
					\draw [fill=blue] (5,0)  circle (3.5pt);
					\draw [fill=red] (4.6,-2.4) node [above] {\large $u$} circle (3.5pt);
					\draw [fill=blue] (5.4,-2.4) circle (3.5pt);
				\end{scriptsize}
			\end{tikzpicture}
			\caption{$u$ is contained in $T^3$ }
		\end{subfigure}	
		\caption{Different positions of $u$ in a  generalized triangle }\label{fig:Tcover}
	\end{figure}
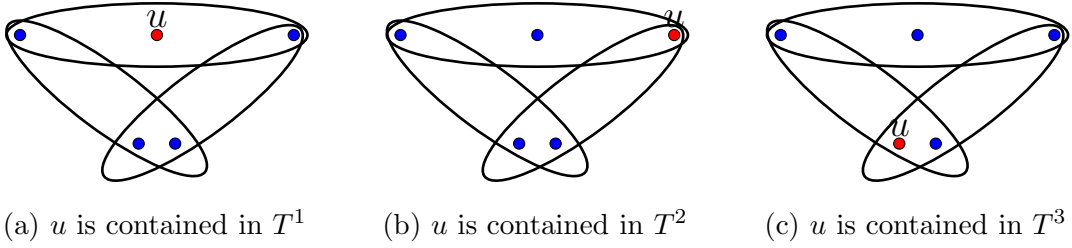
	
	\begin{thm}\label{ttt2}
		For $n \geq 5$, we have:
		$$c_2(n,T)=
		\begin{cases}
			1, \text{when}~ n\in [5,10]\\
			2, \text{when}~ n\geq 11~ and~ n-1 \equiv 0 \pmod {3}\\
			1, \text{when}~n\geq 11~ and~  n-1 \equiv 1,2 \pmod {3}\\
		\end{cases}.$$
	\end{thm}
	
	\begin{thm} \label{ccc}
		For $n \geq 5$, we have:
		\begin{enumerate}
			\item[(i)]
			$\frac{n^2}{9}\leq  c_1(n,T) \leq \frac{n^2}{6}+\frac{5}{6}n-3. $
			\item[(ii)]
			If $G$ is  an $n$-vertex 3-graph  satisfying that $\delta_1(G)> \frac{n^2}{6}+\frac{5}{6}n-3 $, then for any vertex $u$ in $G$, there is a generalized triangle $T^1$ or $T^2$ covering $u$. 
			\item[(iii)]
			If $G$ is  an $n$-vertex 3-graph  satisfying that  $\delta_1(G)> \frac{n^2}{4}+\frac{1}{4}n-2$, then for any vertex $u$ in $G$, there are generalized triangles $T^1$ and $T^2$ covering $u$. 
			\item[(iv)]
			If  $G$ is  an $n$-vertex 3-graph  satisfying that  $\delta_1(G)> \frac{\sqrt{5}-1}{4}n^2+O(n)  $, then for any vertex $u$ in $G$, there are generalized triangles $T^1$, $T^2$ and $T^3$ covering $u$. 
		\end{enumerate}
	\end{thm}
	\par 
	Now we pay attention to some trees covering problems. For $k\geq 2$, let the $3$-graph $S_k$ be the $k$-star with the vertex set $\{v_0,v_1,v_2...,v_{2k-1},v_{2k}\}$ and the edge set $\{\{v_0,v_1,v_2\},\{v_0,v_3,v_4\},...,\{v_0,v_{2k-1},v_{2k}\}\}$. Let $v_0$ be the center of $S_k$. For $k\geq 2$, let the $3$-graph $P_k$ be the the linear $k$-path with the vertex set $\{v_1,v_2...,v_{2k},v_{2k+1}\}$ and the edge set $\{\{v_1,v_2,v_3\}\,\{v_3,v_4,v_5\},...,\{v_{2k-1},v_{2k},v_{2k+1}\}\}$. In this paper, we consider the $F$-covering problem when $F$ is the $k$-star $S_k$ or the linear $k$-path $P_k$. \par  
	
	When $k=2$, the $2$-star $S_2$ is the linear $2$-path $P_2$. %Let $P_2$ be the linear 2-path with the vertex set $\{v_0,v_1,v_2,v_3,v_4\}$ and the edge set $\{\{v_0,v_1,v_2\},\{v_0,v_3,v_4\}\}$. 
	Figure \ref{fig:P_2} is an example of the linear $2$-path $P_2$. We determine the exact value of $c_2(n,P_2)$ and $c_1(n,P_2)$. The results  on the linear $2$-path covering or the $2$-star covering are as follows. \par

	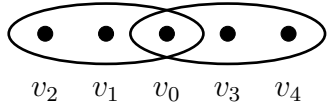
\begin{figure}[!htbp]
		\centering
		\begin{tikzpicture}[scale=0.8]
			\draw [line width=1.pt] (1.,0.) ellipse (1.6cm and 0.5cm);
			\draw [line width=1.pt] (3.,0.) ellipse (1.6cm and 0.5cm);
			\draw[color=black,fill=black] (0.,0.) node[below=0.5] {$v_2$}circle (3.5pt);
			\draw[color=black,fill=black] (1.,0.) node[below=0.5] {$v_1$}circle (3.5pt);
			\draw[color=black,fill=black] (2.,0.) node[below=0.5] {$v_0$}circle (3.5pt);
			
			\draw[color=black,fill=black] (3.,0.) node[below=0.5] {$v_3$}circle (3.5pt);
			\draw[color=black,fill=black] (4.,0.) node[below=0.5] {$v_4$}circle (3.5pt);
			
		\end{tikzpicture}
		\caption{Linear 2-path $P_2$ }\label{fig:P_2}
	\end{figure}

	\begin{thm}\label{s22}
		For $n\geq 5$, we have $c_2(n,P_2)=0$.
	\end{thm}
	\begin{thm}\label{s21}
		For $n\geq 8$, we have $c_1(n,P_2)= 3$. 
	\end{thm}
What's more, we determine the codegree threshold for the property of a $3$-graph $G$ that for any vertex $u\in V(G)$ we can find a linear $2$-path $P_2$ with the center $u$.
\begin{thm}\label{pp2}
	If $G$ is an $n$-vertex $3$-graph satisfying that $n\geq 5$ and $\delta_2(G)\geq 2$, then for any vertex $u\in V(G)$, we can find $4$ vertices $p,q,s,t$ where $\{u,p,q\}$ and $\{u,s,t\}$ form a linear $2$-path $P_2$ covering $u$. 
\end{thm} 
%When $k=3$, let $S_3$ be the $3$-star with the vertex set $\{v_0,v_1,v_2,v_3,v_4,v_5,v_6\}$ and the edge set $\{\{v_0,v_1,v_2\},\{v_0,v_3,v_4\},\{v_0,v_5,v_6\}\}$.
 Through exploring the structure of graphs without some specific matchings in Theorem \ref{sgBt}, we obtain the following result on the $3$-star $S_3$-covering.
	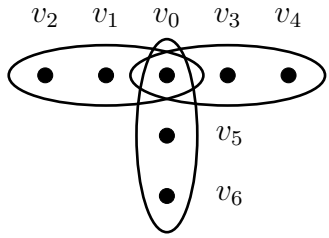
\begin{figure}[!htbp]
	\centering
	\begin{tikzpicture}[scale=0.8]
		\draw [line width=1.pt] (1.,0.) ellipse (1.6cm and 0.5cm);
		\draw [line width=1.pt] (3.,0.) ellipse (1.6cm and 0.5cm);
		\draw [line width=1.pt] (2.,-1.) ellipse (0.5cm and 1.6cm);
		
		\draw[color=black,fill=black] (0.,0.) node[above=0.5] {$v_2$}circle (3.5pt);
		\draw[color=black,fill=black] (1.,0.) node[above=0.5] {$v_1$}circle (3.5pt);
		\draw[color=black,fill=black] (2.,0.) node[above=0.5] {$v_0$}circle (3.5pt);
		
		\draw[color=black,fill=black] (3.,0.) node[above=0.5] {$v_3$}circle (3.5pt);
		\draw[color=black,fill=black] (4.,0.) node[above=0.5] {$v_4$}circle (3.5pt);
		\draw[color=black,fill=black] (2.,-1.) node[right=0.5] {$v_5$}circle (3.5pt);
		\draw[color=black,fill=black] (2.,-2.) node[right=0.5] {$v_6$}circle (3.5pt);
		
	\end{tikzpicture}
	\caption{The $3$-star $S_3$ }\label{fig:S_3}
\end{figure}

\begin{thm}\label{s32}
	If $H$ is an $n$-vertex 3-graph  satisfying $n\geq 7$ and $\delta_2(H)\geq 2 $, then for any vertex $u\in V(H)$, there is a $3$-star $S_3$ covering $u$.
\end{thm}
By Theorem \ref{s32}, we can directly get the following corollary.
\begin{cor}
	For $n\geq 7$, $c_2(n,S_3)\leq 1.$ 
\end{cor}
What's more, we determine the codegree threshold for the property of a $3$-graph $G$ that for any vertex $u\in V(G)$ we can find a $3$-star $S_3$ with the center $u$.
\begin{thm}\label{S3 center cover}
	If $H$ is an $n$-vertex $3$-graph with $n\geq 7$ and $\delta_2(H)\geq 3$, then for any vertex $u\in V(H)$ we can find a $S_3$ with the center $u$.  
\end{thm}
%For $k\geq 3$, let $S_k$ be the $k$-star with the vertex set $\{v_0,v_1,v_2,...,v_{2k-1},v_{2k}\}$ and the edge set $\{\{v_0,v_1,v_2\},\{v_0,v_3,v_4\},\{v_0,v_{2k-1}v_{2k}\}\}$.
 Using the similar technique in the proof of Theorem \ref{s32} we also give bounds of $c_2(n,S_k)$ and  $c_1(n,S_k)$ for $k\geq 3$.
\par 
\begin{Pro}\label{sk2}
	Let $k$ be an integer with $k\geq 3$. Let $H$ be an $n$-vertex 3-graph  with $n\geq 2k+1$. We have:
	\begin{enumerate}
		\item [(i)] $c_2(n,S_k)\leq  \max \{\frac{4k^2-6k+2}{n-1},k-2-\frac{k^2-nk}{n-1}\}$. 
	\item [(ii)] $c_1(n,S_k)\leq \max \{\binom{2k-1}{2},\binom{n-1}{2}-\binom{n-k}{2}\}.$
	\end{enumerate} 

\end{Pro}
%When $k=3$, let $P_3$ be the linear $3$-path with the vertex set $\{v_0,v_1,v_2,v_3,v_4,v_5,v_6\}$ and the edge set $\{\{v_0,v_1,v_2\},\{v_2,v_3,v_4\},\{v_4,v_5,v_6\}\}$.
 Figure \ref{fig:P_3} is an example of the linear $3$-path $P_3$. We determine the exact value of $c_2(n,P_3)$ and asymptotically determine $c_1(n,P_3)$ as follows.

	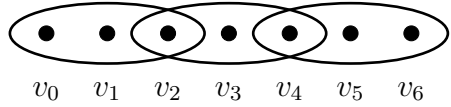
\begin{figure}[!htbp]
	\centering
	\begin{tikzpicture}[scale=0.8]
		\draw [line width=1.pt] (1.,0.) ellipse (1.6cm and 0.5cm);
		\draw [line width=1.pt] (3.,0.) ellipse (1.6cm and 0.5cm);
		\draw [line width=1.pt] (5.,0.) ellipse (1.6cm and 0.5cm);
		
		\draw[color=black,fill=black] (0.,0.) node[below=0.5] {$v_0$}circle (3.5pt);
		\draw[color=black,fill=black] (1.,0.) node[below=0.5] {$v_1$}circle (3.5pt);
		\draw[color=black,fill=black] (2.,0.) node[below=0.5] {$v_2$}circle (3.5pt);
		
		\draw[color=black,fill=black] (3.,0.) node[below=0.5] {$v_3$}circle (3.5pt);
		\draw[color=black,fill=black] (4.,0.) node[below=0.5] {$v_4$}circle (3.5pt);
		\draw[color=black,fill=black] (5.,0.) node[below=0.5] {$v_5$}circle (3.5pt);
		\draw[color=black,fill=black] (6.,0.) node[below=0.5] {$v_6$}circle (3.5pt);
		
	\end{tikzpicture}
	\caption{Linear 3-path $P_3$ }\label{fig:P_3}
\end{figure}

\begin{thm}\label{p32 exact}
	For $n\geq 8$, we have $c_2(n,P_3)=1$.
\end{thm}
\begin{thm}\label{p31}
	For $n\geq 8$, we have  $n-2 \leq c_1(n,P_3)\leq n+4$.
\end{thm}
What's more, we determine the codegree threshold for the property of a $3$-graph $G$ that for any vertex $u\in V(G)$ we can find a linear $3$-path with the vertex set $\{u,v_1,v_2,v_3,v_4,v_5,v_6\}$ and the edge set $\{\{v_1v_2u\},\{uv_3v_4\},\{v_4v_5v_6\}\}$ covering $u$.
\begin{thm}\label{p32 position 2}
	If $H$ is an $n$-vertex $3$-graph with $n\geq 8$ and $\delta_2(H)\geq 3$, then for any vertex $u\in V(H)$ we can find a $P_3$ with the vertex set $\{u,v_1,v_2,v_3,v_4,v_5,v_6\}$ and the edge set $\{\{uv_1v_2\},\{uv_3v_4\},\{v_4v_5v_6\}\}$ covering $u$.
\end{thm}
%For $k\geq 3$, let $P_k$ be the linear $k$-path with the vertex set $\{v_0,v_1,v_2,...,v_{2k-1},v_{2k}\}$ and the edge set be $\{\{v_0,v_1,v_2\},\{v_2,v_3,v_4\},\{v_{2k-2}v_{2k-1}v_{2k}\}\}$.
 We also give the bounds of $c_2(n,P_k)$  and $c_1(n,P_k)$ for $k\geq 4$ as follows.
\begin{Pro}\label{pk2}
Let $k$ be an integer with $k\geq 4$. We have:
	\begin{itemize}
		\item [(i)]	For $n\geq 2k+1$,  $k-3\leq c_2(n,P_k)\leq 2k-2.$
		\item [(ii)] For $n\geq 4k$, $\max \{n-2, \binom{2k-1}{2}\}\leq c_1(n,P_k)\leq \binom{n-1}{2}-\binom{n-2k+1}{2}.$
	\end{itemize}

\end{Pro}
\par
	The rest of the paper is arranged as follows. In Section 2, we give some extremal constructions and proofs of theorems for generalized triangle covering. And in Sections 3 we give some extremal constructions and proofs of theorems for some trees covering.
%	with minimum codegree or minimum degree, which guarantee the constructions having no generalized triangle $T$-covering. And Sections 3, and 4 are the proofs of Theorems 1 and 2. \wang{And Sections 5, and 6 are the proofs of Theorems 3 and 4.}
	
	\section{Results on generalized triangle covering}
	\subsection{Construction}
		\begin{Constrution} \label{mmm}
		Let $V_1$ be a vertex set. Fix $u\in V_1$, let $V'=V_1\setminus\{u\}$, $E_1=\{u\}\vee \binom{V'}{2}$ which means $E_1$ is a 3-set family and every 3-set from $E_1$ contains $u$ and two other vertices form $V'$. Let $G_1=(V_1,E_1)$ be a 3-graph. 
	\end{Constrution}
	The following observation can be checked directly.
	\begin{ob}\label{vvv}
		$G_1$ is a $3$-graph with $\delta_2(G_1)=1$ and there is no generalized triangle $T$ covering $u$.
	\end{ob}

	\begin{Constrution} \label{cg}
		Let $k$ be an integer with $k\geq 4$. Let $G_2=(V_2,E_2)$ be a 3-graph with $V_2$=$\left\{u\right\} \cup \sum_{i=1}^{k}C_{i}$ where $C_{i}$ is a 3-vertex set for $i \in [1,k]$. $E_2$ consists of two types of edges.  For the first type, edges induced in the vertex set $\left\{u\right\} \cup C_{i} $ form a  $K_4^3$ for any $i \in [1,k]$. For the second type, let $C_{a}, C_{b}$ and $C_{c} $ be any three elements in $\left\lbrace   C_{i} : i \in [1,k] \right\rbrace $. Without loss of generality, we assume  $C_{a}$ is $\left\lbrace v_1,v_2,v_3\right\rbrace$,  $C_{b}$ is $\left\lbrace v_4,v_5,v_6\right\rbrace  $ and $C_{c}$ is $\left\lbrace v_7,v_8,v_9\right\rbrace  $. The edges induced in $C_{a}, C_{b}$ and $C_{c} $ are:
		\begin{equation*}
			\left\{
			\begin{array}{c}
				\{v_1,v_4,v_7\},\{v_2,v_4,v_8\},\{v_3,v_4,v_9\};\\
				\{v_1,v_5,v_8\},\{v_2,v_5,v_9\},\{v_3,v_5,v_7\};\\
				\{v_1,v_6,v_9\},\{v_2,v_6,v_7\},\{v_3,v_6,v_8\};
			\end{array}
			\right\}
		\end{equation*}
		
	\end{Constrution}
	
	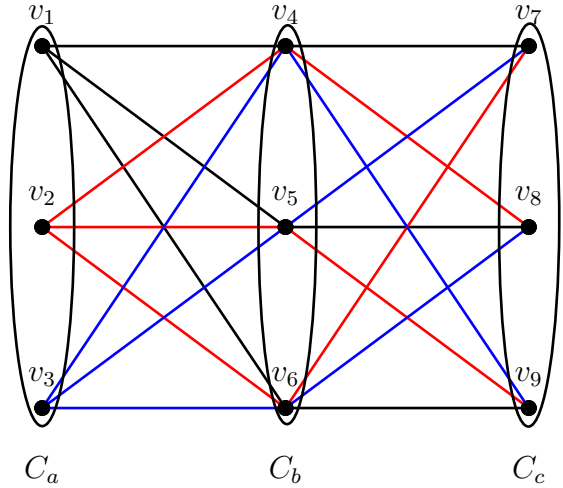
\begin{figure}[!htbp]
		\centering
		\begin{tikzpicture}[scale=0.8]
			\draw [line width=1.pt] (0.,3.)-- (4.,3.);
			\draw [line width=1.pt] (0.,3.)-- (4.,0.);
			\draw [line width=1.pt] (0.,3.)-- (4.,-3.);
			
			\draw [line width=1.pt,color=black] (4.,3.)-- (8.,3.);
			\draw [line width=1.pt] (4.,0.)-- (8.,0.);
			\draw [line width=1.pt] (4.,-3.)-- (8.,-3.);
			
			\draw [line width=1.pt,color=red] (0.,0.)-- (4.,3.);
			\draw [line width=1.pt,color=red] (4.,3.)-- (8.,0.);
			\draw [line width=1.pt,color=red] (0.,0.)-- (4.,0.);
			\draw [line width=1.pt,color=red] (4.,0.)-- (8.,-3.);
			\draw [line width=1.pt,color=red] (0.,0.)-- (4.,-3.);
			\draw [line width=1.pt,color=red] (4.,-3.)-- (8.,3.);
			
			\draw [line width=1.pt,color=blue] (0.,-3.)-- (4.,3.);
			\draw [line width=1.pt,color=blue] (0.,-3.)-- (4.,0.);
			\draw [line width=1.pt,color=blue] (0.,-3.)-- (4.,-3.);
			\draw [line width=1.pt,color=blue] (4.,3.)-- (8.,-3.);
			\draw [line width=1.pt,color=blue] (4.,0.)-- (8.,3.);
			\draw [line width=1.pt,color=blue] (4.,-3.)-- (8.,0.);
			\draw [rotate around={90.:(0.,0.01353479251559942)},line width=1.pt] (0.,0.01353479251559942) ellipse (3.3129244935181337cm and 0.5251939203160844cm);
			\draw [rotate around={89.88896180661943:(4.030770594308117,0.03889161752098947)},line width=1.pt] (4.030770594308117,0.03889161752098947) ellipse (3.305423503513939cm and 0.4755463239593543cm);
			\draw [rotate around={89.780055669578:(8.005452913901342,0.03255241126959792)},line width=1.pt] (8.005452913901342,0.03255241126959792) ellipse (3.339579642760312cm and 0.4946002413591508cm);
			\draw[color=black,fill=black] (0.,0.) node[above] {} circle (3.5pt);
			\draw[color=black,fill=black] (0.,3.) node[above] {} circle (3.5pt);
			\draw[color=black,fill=black] (0.,-3.) node[below=0.5] {$C_a$}circle (3.5pt);
			
			\draw[color=black,fill=black] (4.,0.) node[above] {} circle (3.5pt);
			\draw[color=black,fill=black] (4.,3.) node[above] {} circle (3.5pt);
			\draw[color=black,fill=black] (4.,-3.) node[below=0.5] {$C_b$} circle (3.5pt);
			
			\draw[color=black,fill=black] (8.,0.) node[above] {} circle (3.5pt);
			\draw[color=black,fill=black] (8.,3.) node[above] {} circle (3.5pt);
			\draw[color=black,fill=black] (8.,-3.) node[below=0.5] {$C_c$} circle (3.5pt);
			\draw[color=black,fill=black] (0.,3.) node[above=0.15] {$v_1$}circle (3.5pt);
			\draw[color=black,fill=black] (0.,0.) node[above=0.15] {$v_2$}circle (3.5pt);
			\draw[color=black,fill=black] (0.,-3.) node[above=0.15] {$v_3$}circle (3.5pt);
			
			\draw[color=black,fill=black] (4.,3.) node[above=0.15] {$v_4$}circle (3.5pt);
			\draw[color=black,fill=black] (4.,0.) node[above=0.15] {$v_5$}circle (3.5pt);
			\draw[color=black,fill=black] (4.,-3.) node[above=0.15] {$v_6$}circle (3.5pt);
			
			\draw[color=black,fill=black] (8.,3.) node[above=0.15] {$v_7$}circle (3.5pt);
			\draw[color=black,fill=black] (8.,0.) node[above=0.15] {$v_8$}circle (3.5pt);
			\draw[color=black,fill=black] (8.,-3.) node[above=0.15] {$v_9$}circle (3.5pt);
		\end{tikzpicture}
		\caption{Edges induced in $C_{a}, C_{b}$ and $C_{c} $. } \label{CCC}
	\end{figure}
	
	In Construction \ref{cg}, the subgraph induced in every three elements of $\{C_i\}$ is isomorphic to the 3-graph in Figure \ref{CCC}. And we get the following observation for the Construction \ref{cg}. 
	
	\begin{ob} \label{bbb}
		$G_2$ is a $3$-graph with $\delta_2(G_2)= 2$ and there is no generalized triangle $T$ covering $u$.
	\end{ob}
	
	\begin{proof}
		We first check that $G_2$ has no generalized triangle $T$ covering $u$. If $u$ is covered as the first case in Figure 1, then there is an edge $e_0=\{u,v_1,v_2\}$ such that $v_1,v_2\in C_i$ for $i \in [1,k]$. By the definition of $G_2$, we can not find two vertices $v_3,v_4$ from $V_2\setminus\{u,v_1,v_2\}$ making $\{v_1,v_3,v_4\},\{v_2,v_3,v_4\}$  being edges in $E_2$, a contradiction with the fact that there is a $T^1$ covering $u$. If $u$ is covered as the second case in Figure 1, then there are two edges $e_1=\{u,v_5,v_6\}$ and $e_2=\{u,v_7,v_8\}$ such that $v_5,v_6\in C_i$ and $v_3,v_4\in C_j$ for $i\neq j $ and $ i,j\in [1,k]$. However, there is no edge  induced in any two $C_i's$ in $G_2$ , which means we can not find an edge together with $e_1$ and $e_2$ to form a $T^2$ covering $u$, a contradiction. If $u$ is covered as the third case in Figure 1, then there are two edges $e_3=\{u,v_9,v_{10}\}$ and $e_4=\{u,v_9,v_{11}\}$ such that $v_9,v_{10},v_{11}\in C_i$ for some $ i\in [1,k]$. Actually, there is no vertex $v_{12}$ making $\{v_{10},v_{11},v_{12}\}$ being an edge in $G_2$, a contradiction with the fact that there is a $T^3$ covering $u$.
		
		Next we prove  that $\delta_2(G_2)= 2$. Let $s$,$t$ be any two vertices in $V(G_2)$. We have:
		
		\begin{itemize}
			\item If the two vertices $s,t$ belong to different $C_i's$, then  $d_G(\{s,t\})\geq 2.$
			\item If the two vertices $s,t$ belong to any $C_i$, then  $d_G(\{s,t\})=2.$
			\item If the vertex $s$ is $u$ and the vertex $t$ belongs to any $C_i$, then $d_G(\{s,t\})=2.$
		\end{itemize}\par 
		In conclusion, we have $\delta_2(G_2)= 2$ and there is no generalized triangle $T$ covering $u$.
		
	\end{proof}

	\begin{Constrution}
		Let $G_3=(V_3,E_3)$ be an $n$-vertex 3-graph with $V_3$=$\left\{u\right\} \cup  A_1 \cup  A_2\cup  B$ and $E_3=\left( \{u\}\vee\binom{A_1}{1}\vee\binom{A_2}{1}\right)\cup \left( \binom{A_1}{1}\vee\binom{A_2}{1}\vee\binom{B}{1}\right)\cup \binom{B}{3}$, where $|A_1|=|A_2|=\lceil \frac{n}{3} \rceil.$
	\end{Constrution}
	\begin{figure}[!htbp]
		\centering
		\begin{tikzpicture}[scale=0.6]
			\draw [line width=1.pt] (6.,1.)node [left] {\large $A_1$} circle (1.6cm);
			\draw [line width=1.pt] (14.,1.)node [right] {\large $A_2$} circle (1.6cm);
			\draw [line width=1.pt] (10.,-3.)node [left] {\large $B$} circle (1.6);
			\draw [line width=1.pt,color=blue] (10.,-3.)-- (9,-4);
			\draw [line width=1.pt,color=blue] (9,-4)-- (11.,-4.);
			\draw [line width=1.pt,color=blue] (11.,-4.)-- (10.,-3.);
			
			\draw [line width=1.pt,color=blue] (6,0.6)-- (14,0.6);
			\draw [line width=1.pt,color=blue] (14,0.6)-- (10,-2.4);
			\draw [line width=1.pt,color=blue] (10,-2.4)-- (6,0.6);
			
			\draw [line width=1.pt,color=blue] (10.,4.)-- (6,1.3);
			\draw [line width=1.pt,color=blue] (14,1.3)-- (6,1.3);
			\draw [line width=1.pt,color=blue] (10.,4.)-- (14,1.3);
			\draw[color=black,fill=black] (10.,4.) node[above] {$u$} circle (3.5pt);
		\end{tikzpicture}
		\caption{Construction $3$}
	\end{figure}
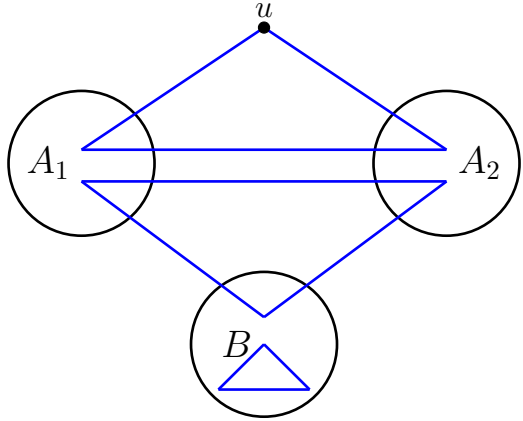
	\begin{ob} \label{nnn}
		$G_3$ is a $3$-graph with $\delta_1(G_3)\geq\frac{n^2}{9}$ and there is no generalized triangle $T$ covering $u$.
	\end{ob}
	\begin{proof}
		We check that $G_3$ has no generalized triangle $T$ covering $u$. If $u$ is covered as the first case in Figure 1, then there is an edge $e_0=\{u,v_1,v_2\}$ such that $v_1,v_2$ are in different $A_i$ for $i=1,2$. By the definition of $G_3$, we can not find two vertices $v_3,v_4$ from $V_3\setminus\{u,v_1,v_2\}$ such that $\{v_1,v_3,v_4\},\{v_2,v_3,v_4\}$ are edges in $E_3$, a contradiction with the fact that there is a $T^1$ covering $u$. If $u$ is covered as the second case in Figure 1, then there are two edges $e_1=\{u,v_5,v_6\}$ and $e_2=\{u,v_7,v_8\}$ such that both $v_5,v_6$ and $v_7,v_8$ are in different $A_i$ for $i=1,2$. However, there is no edge induced in $A_1$ and $A_2$ in $G_3$, which means we can not find an edge together with $e_1$ and $e_2$ to form a $T^2$ covering $u$, a contradiction. If $u$ is covered as the third case in Figure 1, then there are two edges $e_3=\{u,v_9,v_{10}\}$ and $e_4=\{u,v_9,v_{11}\}$ such that  $v_{10},v_{11}$ are in one of $A_i$ and $v_9$ in another one for $i=1,2$. Actually, there is no vertex $v_{12}$ making $\{v_{10},v_{11},v_{12}\}$ being an edge in $G_3$, a contradiction with the fact that there is a $T^3$ covering $u$.
		
		Next we prove that $\delta_1(G_3)\geq\frac{n^2}{9}$. Let $v$ be a vertex from $V(G_3)$.\par
		If $v=u$, then $$d_{G_3}(\{v\})=\lceil \frac{n}{3} \rceil \cdot \lceil \frac{n}{3} \rceil\geq \frac{n^2}{9}.$$\par 
		If $v\in A_1 \cup A_2$, then $$d_{G_3}(\{v\})=\lceil \frac{n}{3} \rceil+\lceil \frac{n}{3} \rceil \cdot (n-1-2\lceil \frac{n}{3} \rceil)\geq\frac{n^2}{9}.$$\par 
		If $v\in B$, then $$d_{G_3}(\{v\})=\lceil \frac{n}{3} \rceil \cdot \lceil \frac{n}{3} \rceil+\binom{n-1-2\lceil \frac{n}{3} \rceil}{2}>\lceil \frac{n}{3} \rceil \cdot \lceil \frac{n}{3} \rceil\geq \frac{n^2}{9}.$$
		Therefore, we have $\delta_1(G_3)\geq\frac{n^2}{9}$.
		
	\end{proof}
	\subsection{The proofs of Theorem \ref{ttt2}}
	We divide the proof of Theorem \ref{ttt2} into two parts according to the value of $n$.
		\subsubsection {When $n\in [5,10]$}
	
	When $n\in [5,10]$, the lower bound of $c_2(n,T)$ can be directly gotten from Observation \ref{vvv}. Therefore, we only need to  prove $c_2(n,T) \leq1$ when $n\in [5,10]$. 
	We assume to the contrary that there is a 3-graph $G$ with $\delta_2(G)\geq 2$ and a vertex $u\in V(G)$ that is not covered by $T$.\par
	Let $y\in V(G)$ be a vertex different from $u$. As $d_G(\{u,y\})\geq 2$, $N_G(\{u,y\})$ has at least two vertices. Considering any two vertices $p$, $q$ from $N_G(\{u,y\})$, we have $N_G(\{p,q\})\subseteq \{u,y\}$. Otherwise, if there is a vertex $t\in N_G(\{p,q\})$ different from $u$ and $y$, then $\{\{p,q,t\},\{u,y,p\},\{u,y,q\}\}$ is a $T$ covering $u$, a contradiction. On the other hand, as $\delta_2(G)\geq 2$, we have $N_G(\{p,q\})= \{u,y\}$. A direct corollary is that $G[\{u,y,p,q\}]=K_4^3$. What's more, any two vertices from $\{u,y,p,q\}$ have codegree 2 in $G$. Otherwise, we can find a $T$ covering $u$, a contradiction.\par 
	Now consider the link graphs of vertices $y,p,q$, we denote them by $G_y,G_p$ and $G_q$, respectively. Let $G_a$ be the 3-graph obtained by deleting the vertices $u,p,q$ (and related edges) from $G_y$, $G_b$ be the 3-graph obtained by deleting the vertices $u,y,q$ (and related edges) from $G_p$, and $G_c$ be the 3-graph obtained by deleting the vertices $u,y,p$ (and related edges) from $G_q$. As $\delta_2(G)\geq 2$ and any two vertices in $\{u,y,p,q\}$ have no co-neighbor out of $\{u,y,p,q\}$, we have $\delta _1(G_a)\geq 2$, $\delta _1(G_b)\geq 2$ and $\delta _1(G_c)\geq 2$. Actually, $G_a, G_b$ and $G_c$ are simple graphs defined on the same vertex set. As $n\in [5,10]$, we have:
	$$ e(G_a)+e(G_b)+e(G_c)\geq \frac{2(n-4)}{2}3 =3(n-4)> \binom{n-4}{2}.$$
	The inequality above implies that there must be at least one edge contained in at least two graphs of $G_a, G_b$ and $G_c$. Without loss of generality, let $\{s,t\}$ be the common edge in $G_a$ and $ G_b$. As a result, we can find a $T=\{\{s,t,y\},\{s,t,p\},\{u,y,p\}\}$ covering $u$, a contradiction.
	
	\subsubsection {When $n\geq 11$}
	We first consider the case for  $n\geq 11$ and $n-1 \equiv 0 \pmod {3}$. By Observation \ref{bbb}, we have $c_2(n,T)\geq 2$ for $n\geq 11$ and $n-1 \equiv 0 \pmod {3}$. Therefore, we only need to  prove $c_2(n,T) \leq2$ for $n\geq 11$ and $n-1 \equiv 0 \pmod {3}$. 
	We suppose to the contrary that there is an $n$-vertex 3-graph $G$ with $\delta_2(G)\geq 3$ for $n\geq 11$ and $n-1 \equiv 0 \pmod {3}$  and a vertex $u\in V(G)$ that is not covered by $T$.\par
	Let $y$ be any other vertex different from $u$ in $G$. As $\delta_2(G)\geq 3$, we have $d_G(\{u,y\})\geq 3$. Therefore, we can find two edges containing $\{u,y\}$. Let the two edges be $\{u,y,p\}$ and $\{u,y,q\}$. Considering $d_G(\{p,q\})\geq 3$, we can find a vertex $o$ different from $u$ and $y$, such that $\{o,p,q\}$ forms an edge in $G$. Therefore, we find a generalized triangle $T$ with the edge set $\{\{u,y,p\},\{u,y,q\},\{o,p,q\}\}$ covering $u$, a contradiction.\par 
	Next, we consider the case for $n\geq 11~ \text{and}~  n-1 \equiv 1,2 \pmod {3}$. By Observation \ref{vvv}, we have $c_2(n,T)\geq 1$ for $n\geq 11$ and $n-1 \equiv 1,2 \pmod {3}$. Therefore, we only need to  prove $c_2(n,T) \leq1$ for $n\geq 11$ and $n-1 \equiv 1,2 \pmod {3}$. We suppose to the contrary that there is an $n$-vertex 3-graph $G$ with $\delta_2(G)\geq 2$ for $n\geq 11$ and $n-1 \equiv 1,2 \pmod {3}$ and a vertex $u\in V(G)$ that is not covered by $T$.\par
	Let $v_1\in V(G)$ be a vertex in $V(G)$ different from $u$. As $\delta_2(G)\geq 2$, we have $d_G(\{u,v_1\})\geq 2$ and $N_G(\{u,v_1\})$ has at least two vertices. Considering any two vertices $v_2$, $v_3$ in $N_G(\{u,v_1\})$, we have $N_G(\{v_2,v_3\})\subseteq \{u,v_1\}$. Otherwise, if there is a vertex $h\in N_G(\{v_2,v_3\})$ different from $u$ and $v_1$, then there is a generalized triangle $T$ with the edge set $\{\{v_2,v_3,h\},\{u,v_1,v_2\},\{u,v_1,v_3\}\}$  covering $u$, a contradiction. On the other hand, as $\delta_2(G)\geq 2$, we have $N_G(\{v_2,v_3\})= \{u,v_1\}$. Actually, we find $G[\{u,v_1,v_2,v_3\}]$ is a complete $3$-graph on $4$ vertices. What's more, any two vertices in $\{u,v_1,v_2,v_3\}$ have codegree 2, which means any two vertices in $\{u,v_1,v_2,v_3\}$ have no co-neighbor out of $\{u,v_1,v_2,v_3\}$.
	Otherwise, we can find a $T$ covering $u$, a contradiction.\par
	Let $v_4$ be a vertex in $V(G)$ different from $v_1,v_2,v_3$ and $u$. As $\delta_2(G)\geq 2$, we have $d_G(\{u,v_4\})\geq 2$ and $N_G(\{u,v_4\})$ has at least two vertices. Let $v_5$ and $v_6$ be any two vertices in $N_G(\{u,v_4\})$. The same as the above analysis, we have $N_G(\{v_5,v_6\})=\{u,v_4\}$ and  $G[\{u,v_4,v_5,v_6\}]$ is a complete $3$-graph on $4$ vertices. Continue to consider other vertices in this way. There must exist a lot of 3-vertex sets: $T_1=\{v_1,v_2,v_3\}$,  $T_2=\{v_4,v_5,v_6\}$,...,$T_l=\{v_{3l-2},v_{3l-1},v_{3l}\}$, such that edges induced in the vertex set $T_i\cup \{u\}$ form a complete $3$-graph on $4$ vertices for $i\in[1,l]$.  Apart from these $3l+1$ vertices, there are one or two vertices left since $n-1 \equiv 1,2 \pmod {3}$.
	\begin{itemize}
		\item If there is exactly one vertex left, let it be $a$. As $\delta_2(G)\geq 2$, we have $d_G(\{u,a\})\geq 2$ and $N_G(\{u,a\})$ has at least two vertices. For the vertex in $N_G(\{u,a\})$, it must be a vertex in a $T_i$ for $i\in[1,l]$. Without loss of generality, let $v_{3i}$ in $T_i$  be a vertex from $N_G(\{u,a\})$, which means $\{v_{3i},u,a\}$ is an edge in $G$. Then we find a generalized triangle $T$ with the edge set $\{\{v_{3i},u,a\},\{v_{3i-2},v_{3i-1},v_{3i}\},\{v_{3i-2},\\v_{3i-1},u\}\}$ covering $u$, a contradiction.
		\item If there are two vertices left, let them be $b$ and $c$. As $\delta_2(G)\geq 2$, we have $d_G(\{u,b\})\geq 2$ and $N_G(\{u,b\})$ has at least two vertices. For the vertex in $N_G(\{u,b\})\setminus\{c\}$, it must be a vertex in a $T_i$ for $i\in[1,l]$. Then through the same analysis as the case before for one vertex left, we can find a $T$ covering $u$, a contradiction.
	\end{itemize}
	\par 
	Therefore, we have $c_2(n,T)=1$ for $n\geq 11~ and~  n-1 \equiv 1,2 \pmod {3}$.
	\subsection{The proof of Theorem \ref{ccc}}
	\subsubsection {The proof of (i)}   We can directly get the lower bound of $c_1(n,T)$ from Observation \ref{nnn}. Therefore, it is sufficient to show that every 3-graph $G$ on $n$ vertices with $\delta_1(G)>\frac{n^2}{6}+\frac{5}{6}n-3$ has a $T$-covering. Suppose to the contrary that there is an $n$-vertex 3-graph $G$ with $\delta_1(G)>\frac{n^2}{6}+\frac{5}{6}n-3$  and a vertex $u\in V(G)$ that is not covered by $T$.\par
	Consider an edge $e=\{u,x,y\}$ containing $u$ in $G$. Let $G_x$, $G_y$ and $G_u$ be the link graphs of $x$, $y$ and $u$, respectively. Let $G_{x}^{'}$ be the 3-graph obtained by deleting the vertices $u,y$ (and related edges) from $G_x$, $G_{y}^{'}$ be the 3-graph obtained by deleting the vertices $u,x$ (and related edges) from $G_y$ and $G_{u}^{'}$ be the 3-graph obtained by deleting the vertices $x,y$ (and related edges) from $G_u$.
	Then $G_{x}^{'}$, $G_{y}^{'}$ and $G_{u}^{'}$ are simple graphs defined on the same $n-3$ vertices. As
	$$\delta_1(G)>\frac{n^2}{6}+\frac{5}{6}n-3,$$
	we have
	$$e(G_{x}^{'})+e(G_{y}^{'})+e(G_{u}^{'})>3\cdot (\frac{n^2}{6}+\frac{5}{6}n-3-2(n-3)-1)>\binom{n-3}{2},$$
	which means that there must be a common edge in at least two of $G_{x}^{'}$, $G_{y}^{'}$ and $G_{u}^{'}$. Without loss of generality, let $\{s,t\}$ be the common edge in both $G_{x}^{'}$ and $ G_{y}^{'}$. As a result, we can find a $T=\{\{s,t,x\},\{s,t,y\},\{u,x,y\}\}$ covering $u$, a contradiction.
	%Let the endpoint of the edge be $m$ and $n$. We find $\{\{m,n,x\},\{m,n,y\},\{x,y,u\}\}$ is a $T$ covering $u$, a contradiction. \par 
	\subsubsection {The proof of (ii)}
	Let $G$ be a 3-graph with $\delta_1(G)> \frac{n^2}{6}+\frac{5}{6}n-3 $. We will prove for any vertex $u$, there is a generalized triangle $T^1$ or $T^2$ covering $u$. \par
	 For any vertex 	$u\in V(G)$, as $d_G(\{u\})>1$, there is an edge $\{u,x,y\}$ containing $u$. Let $G_x$, $G_y$ and $G_u$ be the link graphs of $x$, $y$ and $u$, respectively. Let $G_{x}^{'}$ be the 3-graph obtained by deleting the vertices $u,y$ (and related edges) from $G_x$, $G_{y}^{'}$ be the 3-graph obtained by deleting the vertices $u,x$ (and related edges) from $G_y$, and $G_{u}^{'}$ be the 3-graph obtained by deleting the vertices $x,y$ (and related edges) from $G_u$. Then $G_{x}^{'}$, $G_{y}^{'}$ and $G_{u}^{'}$ are simple graphs defined on the same $n-3$ vertices. As $$\delta_1(G)> \frac{n^2}{6}+\frac{5}{6}n-3, $$
	we have
	$$e(G_{x}^{'})+e(G_{y}^{'})+e(G_{u}^{'})>3\cdot (\frac{n^2}{6}+\frac{5}{6}n-3-2(n-3)-1)>\binom{n-3}{2},$$
	which means that there must be a common edge in at least two of $G_{x}^{'}$, $G_{y}^{'}$ and $G_{u}^{'}$. If there is an edge $\{s,t\}$ in both $G_{x}^{'}$ and $ G_{y}^{'}$, we can find a $T^1=\{\{s,t,x\},\{s,t,y\},\{u,x,\\y\}\}$ covering $u$. If there is an edge $\{p,q\}$ in both $G_{x}^{'}$ and $ G_{u}^{'}$, we can find a $T^2=\{\{p,q,x\},\{p,q,u\},\{x,y,u\}\}$ covering $u$. If there is an edge $\{g,h\}$ in both $G_{y}^{'}$ and $ G_{u}^{'}$, we can find a $T^2=\{\{g,h,y\},\{g,h,u\},\{u,x,y\}\}$ covering $u$.\par 
	In conclusion, if $G$ is  a 3-graph  satisfying that $\delta_1(G)> \frac{n^2}{6}+\frac{5}{6}n-3 $, then for any vertex $u\in V(G)$, we can find a generalized triangle $T^1$ or $T^2$ covering $u$. 

	\subsubsection {The proof of (iii)}
	Let $G$ be a 3-graph with $\delta_1(G)> \frac{n^2}{4}+\frac{1}{4}n-2$. We will prove for any vertex $u\in V(G)$, there are generalized triangles $T^1$ and $T^2$ covering $u$. \par
Since $d_G(\{u\})>1$, there is an edge $\{u,x,y\}$ containing $u$. Let $G_x$, $G_y$ and $G_u$ be the link graphs of $x$, $y$ and $u$, respectively. Let $G_{x}^{'}$ be the 3-graph obtained by deleting the vertices $u,y$ (and related edges) from $G_x$, $G_{y}^{'}$ be the 3-graph obtained by deleting the vertices $u,x$ (and related edges) from $G_y$, and $G_{u}^{'}$ be the 3-graph obtained by deleting the vertices $x,y$ (and related edges) from $G_u$.
	Then $G_{x}^{'}$, $G_{y}^{'}$ and $G_{u}^{'}$ are simple graphs defined on the same $n-3$ vertices. As
	$$\delta_1(G)>\frac{n^2}{4}+\frac{1}{4}n-2,$$
	we have
	$$e(G_{x}^{'})+e(G_{y}^{'})>2\cdot (\frac{n^2}{4}+\frac{1}{4}n-2-2(n-3)-1)>\binom{n-3}{2};$$
	$$e(G_{x}^{'})+e(G_{u}^{'})>2\cdot (\frac{n^2}{4}+\frac{1}{4}n-2-2(n-3)-1)>\binom{n-3}{2};$$
	which means that there must be a common edge in both $G_{x}^{'}$ and $G_{y}^{'}$ and a common edge in both $G_{x}^{'}$ and $G_{u}^{'}$. Without loss of generality, let $\{s,t\}$ be the common edge in $G_{x}^{'}$ and $ G_{y}^{'}$ and $\{p,q\}$ be the common edge in $G_{x}^{'}$ and $ G_{u}^{'}$. 
	As a result, we can find a $T^1=\{\{s,t,x\},\{s,t,y\},\{u,x,y\}\}$ covering $u$ and a $T^2=\{\{p,q,x\},\{p,q,u\},\{x,y,u\}\}$ covering $u$.\par 
	In conclusion, if $G$ is  a 3-graph  satisfying that  $\delta_1(G)> \frac{n^2}{4}+\frac{1}{4}n-2$, then for any vertex $u\in V(G)$, there are generalized triangles $T^1$ and $T^2$ covering $u$.\par
	
	\subsubsection {The proof of (iv)}
	Let $G$ be a 3-graph with $\delta_1(G)> \frac{\sqrt{5}-1}{4}n^2+O(n) $.\par 
	We first prove that for any vertex $u$ in $V(G)$, there is a generalized triangle $T^3$ covering $u$. Suppose to the contrary that there is an $n$-vertex 3-graph $G$ with $\delta_1(G)>\frac{\sqrt{5}-1}{4}n^2+O(n)$  and a vertex $u\in V(G)$ that is not covered by a $T^3$. \par
	Let $G_u$ be the link graph of $u$.  For any vertex $x\in V(G_u)$, let $N=N_{G_u}(\{x\})$ be the neighbor set of $x$ in $G_u$. Considering any vertex $x\in V(G_u)$ with $d_{G_u}(\{x\})\geq 2$, let $y,z$ be any two vertices in $N$. As there is no $T^3$ covering $u$, for any vertex $f\in V(G_u)-\{x,y,z\}$, we have $\{y,z,f\}$ is not an edge in $G$. We call these triples like $\{y,z,f\}$ as the non-edges.
	% For vertex $x\in V(G_u)$ with $d_{G_u}(\{x\})\leq 1$, actually there are not many such vertex as the edges in $G_u$ is $O(n^2)$. 
	Considering the fact that every vertex $x\in V(G_u)$ with $d_{G_u}(\{x\})\geq 2$ is contained in no $T^3$, we collect the triples of these non-edges for every vertex $x\in V(G_u)$. And we denote  the multiset of the non-edges by $M$. Actually, every non-edge in $M$ can be repeated at most $3\cdot \Delta_{G_u}$ times, where $\Delta_{G_u}$ denotes the maximum degree of $G_u$.
	As a result, there are at least $m$ non-edges in $G$, in which 
	$m\geq \frac{|M|}{3\cdot \Delta_{G_u}}.$
	Counting the size of $M$, we have
	\begin{align*}
		|M|&=\sum\limits_{x\in V(G_u),d_{G_u}(\{x\})\geq 2} (n-4)\binom{d_{G_u}(\{x\})}{2}\\
		&=\frac{n-4}{2}\cdot	\sum\limits_{x\in V(G_u),d_{G_u}(\{x\})\geq 2} (d_{G_u}^2(\{x\})-d_{G_u}(\{x\}))\\
		&= \frac{n-4}{2}\cdot	\sum\limits_{x\in V(G_u),d_{G_u}(\{x\})\geq 2} ((d_{G_u}(\{x\})-\frac{1}{2})^2-\frac{1}{4})\\
		&\geq \frac{n-4}{2}\cdot	\sum\limits_{x\in V(G_u),d_{G_u}(\{x\})\geq 2} (d_{G_u}(\{x\})-\frac{1}{2})^2-\frac{n\cdot(n-4) }{8}.\\
	\end{align*}
	By handshaking theorem, we have:
	\begin{align*}
		\sum\limits_{v\in V(G_u),d_{G_u}(\{x\})\geq 2} d_{G_u}(\{x\})&\geq 2\cdot e(G_u)-n\\
		&\geq 2\cdot \delta_1(G) -n.
	\end{align*}
	And by H\"{o}lder inequality, we have:
	\begin{align*}
		\sum\limits_{x\in V(G_u),d_{G_u}(\{x\})\geq 2} (d_{G_u}(\{x\})-\frac{1}{2})^2&\geq \frac{(\sum\limits_{x\in V(G_u),d_{G_u}(\{x\})\geq 2} (d_{G_u}(\{x\})-\frac{1}{2}))^2}{\sum\limits_{x\in V(G_u),d_{G_u}(\{x\})\geq 2} 1}\\
		&\geq \frac{(2\cdot \delta_1(G) -n-\frac{1}{2}\cdot n)^2}{n}\\
		&=	\frac{4\cdot \delta_1(G)^2-6\cdot n\cdot \delta_1(G)+\frac{9}{4}\cdot n^2}{n}.
	\end{align*}
	
	Hence,
	\begin{align}
		m&\geq \frac{ \frac{n-4}{2}\cdot	\sum\limits_{x\in V(G_u),d_{G_u}(\{x\})\geq 2} (d_{G_u}(\{x\})-\frac{1}{2})^2-\frac{n\cdot(n-4) }{8}}{3n}\notag\\
		&\geq \frac{n-4}{2\cdot 3n}\cdot \frac{4\cdot \delta_1(G)^2-6\cdot n\cdot \delta_1(G)+\frac{9}{4}\cdot n^2}{n} -\frac{n\cdot(n-4) }{8\cdot 3n}\notag\\
		&\geq \frac{2(n-4)}{3n^2}\delta_1^2-\frac{n-4}{n}\delta_1-\frac{n-4 }{3}\label{ppp}.
	\end{align}
	
	On the other hand, as $e(G_u)=d_1(\{u\})\geq \delta_1(G)$, there are at most $m$ non-edges in $G$, in which
	\begin{equation}\label{ooo}
		m\leq \binom{n}{3}-\frac{\delta_1(G)\cdot n}{3}.
	\end{equation}
	As $\delta_1(G)> \frac{\sqrt{5}-1}{4}n^2+O(n)$, we have a contradiction between (\ref{ppp}) and (\ref{ooo}), which means that every vertex in $G$ can be covered by a $T^3$ if $ \delta_1(G)> \frac{\sqrt{5}-1}{4}n^2+O(n).$ As 
	$$\frac{\sqrt{5}-1}{4}n^2+O(n)>\frac{n^2}{4}+\frac{1}{4}n-2,$$
	combining with the proof of (iii) we have for any vertex in $G$, we can find generalized triangles $T^1$ and $T^2$ covering it if $ \delta_1(G)> \frac{\sqrt{5}-1}{4}n^2+O(n)$.\par 
	In conclusion, if $G$ is a 3-graph  satisfying that $\delta_1(G)> \frac{\sqrt{5}-1}{4}n^2+O(n)  $, then for any vertex $u\in G$, there are generalized triangles $T^1$, $T^2$ and $T^3$ covering $u$. \par

	\section{Results on some trees covering}
	\subsection{Star covering}
	\subsubsection{Construction}
		\begin{Constrution}
			Let $V_4$ be an $n$-vertex set with $n\geq 8$ and $E_4$ be a $3$-element set. Let $A$ be a $4$-vertex subset of $V_4$ and $B$ be the remain vertex set $V\setminus A$. Let $E_4=\binom{A}{3}\cup \binom{B}{3}$. Let $G_4=(V_4,E_4)$ be a 3-graph. Actually, we have  $G_4=K_4^3\cup K_{n-4}^3.$
		\end{Constrution}
		\begin{figure}[!htbp]
			\centering
			\begin{tikzpicture}[scale=0.6]
				\draw [line width=1.pt] (0.,0.)node [below=2] {\large $A$} circle (2cm);
				\draw [line width=1.pt] (6.,0.)node [below=2] {\large $B$} circle (3cm);
				
				\draw [line width=1.pt,color=blue] (-1,-1)-- (1,-1);
				\draw [line width=1.pt,color=blue] (-1,-1)-- (0,1);
				\draw [line width=1.pt,color=blue] (1,-1)-- (0,1);
				
				\draw [line width=1.pt,color=blue] (4,-1.)-- (8,-1);
				\draw [line width=1.pt,color=blue] (4,-1)-- (6,1.5);
				\draw [line width=1.pt,color=blue] (8.,-1.)-- (6,1.5);
			\end{tikzpicture}
			\caption{Construction $4$}
		\end{figure}
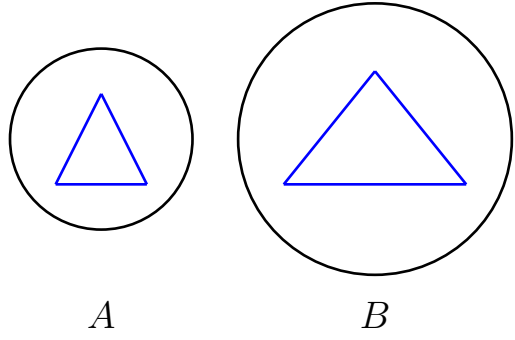
		\begin{ob} \label{k4ob}
			$G_4$ is a $3$-graph with $\delta_1(G_4)=3$ and there is no $P_2$ in $G_4$ covering vertices in $A$.		\end{ob}
			\begin{proof}
				As $|A|=4$, there is no $P_2$ in the induced graph $G[A]$. And $A$ and $B$ are disconnected, so there is no $P_2$ covering the vertices in $A$.\par 
				Now we prove that $\delta_1(G_4)=3$. Let $v$ be a vertex in $V(G_4)$.\par
				If $v\in A$, then $$d_{G_4}(\{v\})= \binom{4-1}{2}=3.$$\par 
				If $v\in B$, as $n\geq 8$ we have $$d_{G_4}(\{v\})=\binom{n-5}{2}\geq 3 .$$\par 
				Therefore, we have $\delta_1(G_4)=3$.
				
			\end{proof}

\begin{Constrution}
	Let $V_5$ be a vertex set. Fix two vertices $u, v_0 \in V_5$, let $V_{5}'=V_5\setminus \{u,v_0\}$ and $E_5=\{u, v_0\}\vee \binom{V_{5}'}{1}\cup \{ v_0\}\vee \binom{V_{5}'}{2}.$ Let $G_5=(V_5,E_5)$ be a $3$-graph. 
\end{Constrution}
\begin{ob}\label{pp22}
	$G_5$ is a $3$-graph with $\delta_2(G_5)= 1$. For the vertex $u$, we can not find $4$ vertices $p,q,s,t$ such that $\{u,p,q\}$ and $\{u,s,t\}$ form a linear $2$-path $P_2$ covering $u$.
\end{ob}
\begin{proof}
	Considering any two vertices $v_1,v_2\in V(G_5)$, we have:
	\begin{itemize}
		\item If $v_1=u$ and $v_2=v_0$, we have $d_{G_5}(\{u,v_0\})\geq 1$.
		\item If $v_1=u$ and $v_2\in V_5'$, we have $d_{G_5}(\{u,v_2\})=1$.
		\item If $v_1=v_0$ and $v_2\in V_5'$, we have $d_{G_5}(\{u,v_0\})\geq 1$.
		\item If $v_1,v_2\in V_5'$, we have $d_{G_5}(\{u,v_0\})= 1$.
	\end{itemize}
	Hence, we have $\delta_2(G_5)= 1$.
	\par 
	Let $G_u$ be the link graph of $u$. As $G_u$ is a star, we can not find $4$ vertices $p,q,s,t$ where $\{u,p,q\}$ and $\{u,s,t\}$ form a linear $2$-path $P_2$ covering $u$.
\end{proof}
\par
\begin{Constrution}
	Let $V_6$ be a vertex set. Fix three vertices $u,a,b$ in $V_6$, let $V_6'$ be $V_6\setminus \{u,a,b\}$. Let $E_6=\{u,a,b\}\cup \{u,a\}\vee \binom{V_6'}{1}\cup \{u,b\}\vee \binom{V_6'}{1}\cup \binom{V_6\setminus\{u\}}{3}$. Let $G_6=(V_6,E_6)$ be a $3$-graph.
\end{Constrution}
\begin{ob}\label{S3center lower bound}
	$G_6$ is a $3$-graph with $\delta_2(G_6)= 2$ and there is no $S_3$ with the center $u$ covering $u$.
\end{ob}
\begin{proof}
	Considering any two vertices $v_1,v_2\in V_6$, we have:
	\begin{itemize}
		\item If $v_1=u$ and $v_2=a$, we have $d_{G_6}(\{u,a\})\geq 2$.
		\item If $v_1=u$ and $v_2=b$, we have $d_{G_6}(\{u,b\})\geq 2$.
		\item If $v_1=u$ and $v_2\in V_6'$, we have $d_{G_6}(\{u,v_2\})=2$.
		\item If $v_1\in V_6\setminus\{u\}$ and $v_2\in V_6\setminus\{u\}$, we have $d_{G_6}(\{v_1,v_2\})\geq 2$.
		
	\end{itemize}
	\par 
	Hence, we have $\delta_2(G_6)= 2$. Now we only need to prove there is no $S_3$ with the center $u$.
	\par
	Let $G_u$ be the link graph of $u$. Then we find $G_u$ is the book graph which has no $3$-matching. Hence there is no $S_3$ with the  center $u$. 
\end{proof}
\subsubsection{The proof of Theorem \ref{s22}}
As $c_2(n,P_2)\geq 0$, we only need to prove $c_2(n,P_2)$ could not be larger than 0. Suppose to the contrary that there is a 3-graph $G$ with $\delta_2(G)\geq 1$ and a vertex $u\in V(G)$ that is not covered by $P_2$.\par
Let $v_1$ be a vertex different from $u$ in $G$. As $d(\{u,v_1\})\geq 1$, there is a vertex $v_2$ making $\{u,v_1,v_2\}$ being an edge in $G$. Consider another vertex $v_3$ in $G$, as $d(\{u,v_3\})\geq 1$ and there is no $P_2$ containing $u$, we have $N_G(\{u,v_3\})\subseteq \{v_1,v_2\}$.\par
If $d(\{u,v_3\})> 1$, we have $N_G(\{u,v_3\})= \{v_1,v_2\}.$ Let $v_4$ be a vertex in $V{(G)}\setminus \{u,v_1,v_2,v_3\}$. As $d(\{v_1,v_4\})\geq 1$ and there is no $P_2$ containing $u$, we have $N_G(\{v_1,v_4\})\\ \subseteq \{u,v_2,v_3\}$. If $u\in N_G(\{v_1,v_4\})$, then we find a $P_2$ with the edge set $\{\{u,v_2,v_3\},\{u,\\v_1,v_4\}\}$ containing $u$, a contradiction.
If $v_2\in N_G(\{v_1,v_4\})$, then we find a $P_2$ with the edge set $\{\{u,v_2,v_3\},\{v_1,v_2,v_4\}\}$ containing $u$, a contradiction.
If $v_3\in N_G(\{v_1,v_4\})$, then we find a $P_2$ with the edge set $\{\{u,v_2,v_3\},\{v_1,v_3,v_4\}\}$ containing $u$, a contradiction. As $N_G(\{v_1,v_4\})\neq\emptyset $, we have a contradiction with $N_G(\{v_1,v_4\})\subseteq \{u,v_2,v_3\}$.
\par
If $d(\{u,v_3\})= 1$, without loss of generality, let $N_G(\{u,v_3\})= \{v_1\}$.  Let $v_5$ be a vertex in $V(G)\setminus \{u,v_1,v_2,v_3\}$. As $d(\{v_3,v_5\})\geq 1$ and there is no $P_2$ containing $u$, we have $N_G(\{v_3,v_5\})\subseteq \{u,v_1,v_2\}$. If $u\in N_G(\{v_3,v_5\})$, then we find a $P_2$ with the edge set $\{\{u,v_3,v_5\},\{u,v_1,v_2\}\}$ containing $u$, a contradiction.
If $v_1\in N_G(\{v_3,v_5\})$, then we find a $P_2$ with the edge set $\{\{u,v_1,v_2\},\{v_1,v_3,v_5\}\}$ containing $u$, a contradiction.
If $v_2\in N_G(\{v_3,v_5\})$, then we find a $P_2$ with the edge set $\{\{u,v_1,v_3\},\{v_2,v_3,v_5\}\}$ containing $u$, a contradiction. As $N_G(\{v_3,v_5\})\neq\emptyset $, we have a contradiction with $N_G(\{v_3,v_5\})\subseteq \{u,v_1,v_2\}$.\par
In conclusion, we have $c_2(n,P_2)=0$.
\subsubsection{The proof of Theorem \ref{s21}}
	The lower bound of $c_2(n,P_2)$ can be directly get from Observation \ref{k4ob}. Therefore, we only need to  prove $c_2(n,T) \leq3$ when $n\geq 8$.
Suppose to the contrary that there is a 3-graph $G$ with $\delta_1(G)\geq 4$ and a vertex $u\in V(G)$ that is not covered by $P_2$. Let $G_u$ be the link graph of $u$.\par
As there is no $P_2$ covering $u$, there is no subgraph $G'$ with the vertex set $\{a,b,c,d\}$ and the edge set $\{\{a,b\},\{c,d\}\}$ in $G_u$. Otherwise, we can find a $P_2$ with the edge set $\{\{a,b,u\},\{c,d,u\}\}$ covering $u$.
Cause $\delta_1(G)\geq 4$, we have there are at least $4$ edges in $G_u$. Therefore, there must be a star S with the vertex set $\{v_1,v_2,v_3\}$ and the edge set $\{\{v_1,v_2\},\{v_2,v_3\}\}$
in $G_u$. We claim that $N_G(\{v_1\})\subseteq \{uv_2,uv_3,v_2v_3\}$. Otherwise, if there are two vertices $x,y$ in $V(G)\setminus\{v_1,v_2,v_3\}$ making $\{xy\}\in N_G(\{v_1\})$, then we find a $P_2$ with the edge set $\{\{u,v_1,v_2\},\{v_1,x,y\}\}$ covering $u$, a contradiction. If there is a vertex $s$ in $V(G)\setminus\{v_1,v_2,v_3\}$ making $\{su\}\in N_G(\{v_1\})$, then we find a $P_2$ with the edge set $\{\{u,v_1,s\},\{u,v_2,v_3\}\}$ covering $u$, a contradiction. If there is a vertex $t$ in $V(G)\setminus\{v_1,v_2,v_3\}$ making $\{tv_2\}\in N_G(\{v_1\})$, then we find a $P_2$ with the edge set $\{\{t,v_1,v_2\},\{u,v_2,v_3\}\}$ covering $u$, a contradiction. If there is a vertex $q$ in $V(G)\setminus\{v_1,v_2,v_3\}$ making $\{qv_3\}\in N_G(\{v_1\})$, then we find a $P_2$ with the edge set $\{\{q,v_1,v_3\},\{u,v_1,v_2\}\}$ covering $u$, a contradiction. Therefore, we have $N_G(\{v_1\})\subseteq \{uv_2,uv_3,v_2v_3\}$. However, that contradicts $\delta_1(G)\geq 4$.\par
 In conclusion, we have  $c_1(n,P_2)= 3$ when $n\geq 8$.

\subsubsection{The proof of Theorem \ref{pp2}}
Suppose to the contrary that there is a $3$-graph $G$ with $\delta_2(G)\geq 2$ and a vertex $u$, such that there is no $4$ vertices $p,q,s,t$ making $\{u,p,q\}$ and $\{u,s,t\}$ forming a linear $2$-path $P_2$ covering $u$. Let $G_u$ be the link graph of $u$.\par 
As $\delta_2(G)\geq 2$, we have $\delta(G_u)\geq 2$. Since there is no $4$ vertices $p,q,s,t$ such that $\{u,p,q\}$ and $\{u,s,t\}$ form a linear $2$-path $P_2$ covering $u$, $G_u$ has no 2-matching. We claim that $G_u$ has only one component. Otherwise, as $\delta(G_u)\geq 2$, there is a 2-matching in $G_u$. As $|V(G_u)|=n-1$ and $n\geq 5 $, we have $G_u$ must be a star. However, the leaves in $G_u$ only have degree $1$, a contradiction with $\delta(G_u)\geq 2$. 
\par
Hence, if $G$ is an $n$-vertex $3$-graph satisfying that $n\geq 5$ and $\delta_2(G)\geq 2$, then for any vertex $u\in V(G)$, we can find $4$ vertices $p,q,s,t$ such that $\{u,p,q\}$ and $\{u,s,t\}$ form a linear $2$-path $P_2$ covering $u$. 
\par 
What's more, by Observation \ref{pp22}, we have the inequality in Theorem \ref{pp2} is sharp.
\subsubsection{The proof of Theorem \ref{s32}}
%For $k\geq 2$, let $k$-matching  be $k$ vertex-disjoint edges.
For a positive integer $t$, the book graph $B_t$ is the graph obtained by the amalgamation of $t$ triangles along the same edge. Let $B_t^-$ be the graph obtained by deleting the common edge from the book graph $B_t$. Before we prove Theorem \ref{s32}, we explore the structure of graphs without some specific matchings and obtain the following result.
\begin{thm}\label{sgBt}
	Let $G$ be an $n$-vertex simple graph with $n\geq 7$. If $G$ has no $3$-matching and  $\delta(G)\geq 2$, then $G$ be the book graph $B_{n-2}$ or the graph $B_{n-2}^-$.
\end{thm}
\begin{proof}[Proof of Theorem \ref{sgBt}]
	Consider the components of $G$.
	\par 
	If $G$ has more than $2$ components, as $\delta(G)\geq 2$, every component has no isolated vertices. Hence we can choose one edge in each component and then we can find at least three vertex-disjoint edges, a contradiction with $G$ has no $3$-matching. 
	\par  
	If $G$ has $2$ components, let them be $G_1,G_2$. As $\delta(G)\geq 2$, every component has no isolated vertices. We can choose one edge in $G_1$. As there is no $3$-matching in $G$, there is no 2-matching in $G_2$. Hence $G_2$ must be a $3$-cycle or a star. If $G_2$ is a star, then the leaves of $G_2$ only have degree $1$, a contradiction with $\delta(G)\geq 2$. If $G_2$ be a $3$-cycle, as $n\geq 7$, $G_1$ has at least $4$ vertices. And $G_1$ must be a star because $G_1$ also can not have $2$-matching. Hence the leaves of $G_1$ only have degree $1$, a contradiction with $\delta(G)\geq 2$. 
	\par 
	Now we only need to consider the case that $G$ is a connected graph. We claim that $G$ must have a cycle. Otherwise, let $P$ be the longest path in $G$ with endpoints $u,v$. As $\delta(G)\geq 2$, $u$ must have at least two neighbors. Since $P$ is the longest path, the neighbors of $u$ must in $V(P)$, which makes a cycle in $G$, a contradiction.
	\par
	Now let $C$ be the longest cycle in $G$. Consider the length of $C$.
	\begin{itemize}
		\item If the length of $C$ is more than $5$, then there exists a $3$-matching in $G$, a contradiction with $G$ has no $3$-matching.
		\item  If the longest cycle $C$ is a $5$-cycle, let the vertex set $V(C)$ be $\{v_1,v_2,v_3,v_4,v_5\}$ and the edge set $E(C)$ be $\{\{v_1v_2\},\{v_2v_3\},\{v_3v_4\},\{v_4v_5\},\{v_5v_1\}\}$. As $G$ is a connected $n$-vertex graph and $n\geq 7$, there is at least one vertex in $V(C)$ sending at least one edge to $V(G)\setminus V(C)$. Without loss of generality, let $v_1$ send one edge to $v_6\in V(G)\setminus V(C).$ Then we find a $3$-matching $\{\{v_1v_6\},\{v_2v_3\},\{v_4v_5\}\}$, a contradiction.
		\item If the longest cycle $C$ is a $3$-cycle, let the vertex set $V(C)$ be $\{v_1,v_2,v_3\}$ and the edge set $E(C)$ be $\{\{v_1v_2\},\{v_2v_3\},\{v_3v_1\}\}$. As $G$ is connected, there must be some vertices in $\{v_1,v_2,v_3\}$ sending edges into $V(G)\setminus \{v_1,v_2,v_3\}$. Without loss of generality, let $v_1$ adjacent to $v_4\in V(G)\setminus\{v_1,v_2,v_3\}$. As $\delta(G)\geq 2$, we have $v_4$ has at least $2$ neighbors. If $v_4$ is adjacent to $v_2$ or $v_3$, then there is a $4$-cycle in $G$, a contradiction with the longest cycle is a $3$-cycle. Hence $v_4$ has a neighbor in $V(G)\setminus \{v_1,v_2,v_3,v_4\}$, let it be $v_5$. Since $d(v_5)\geq 2$, $v_5$ has at least one neighbor other than $v_4$. If $v_5$ has a neighbor in $\{v_2,v_3\}$, then there is a $4$-cycle in $G$, a contradiction with the longest cycle is $3$-cycle. If $v_5$ has a neighbor in $V(G)\setminus \{v_1,v_2,v_3,v_4,v_5\}$, let it be $v_6$, then there is a $3$-matching $\{\{v_2v_3\},\{v_1v_4\},\{v_5v_6\}\}$ in $G$, a contradiction with $G$ has no $3$-matching. If $v_1$ is a neighbor of $v_5$, then there must be some vertices in $\{v_1,v_2,v_3,v_4\}$ sending edges into $V(G)\setminus \{v_1,v_2,v_3,v_4,v_5\}$ since $G$ is connected. If $v_1$ is adjacent to $v_7\in V(G)\setminus\{v_1,v_2,v_3,v_4,v_5\}$, then there is a $3$-matching $\{v_1v_7,v_2v_3, v_4v_5\}$, a contradiction. If there is a vertex in $\{v_2,v_3,v_4\}$ adjacent to $v_8\in V(G)\setminus\{v_1,v_2,v_3,v_4,v_5\}$, let $v_2$ be adjacent to $v_8$, then there is a $3$-matching $\{\{v_2v_8\},\{v_1v_3\},\{v_4v_5\}\}$, a contradiction. 
	\end{itemize} 
	\par 
	Hence, the longest cycle $C$ must be a $4$-cycle. Let the vertex set $V(C)$ be
	$\{v_1,v_2,v_3,v_4\}$ and the edge set $E(C)$ be $\{\{v_1v_2\},\{v_2v_3\},\{v_3v_4\},\{v_4v_1\}\}$.
	\par 
	As $G$ is connected, there must be some vertices in $\{v_1,v_2,v_3,v_4\}$ sending edges into $V(G)\setminus\{v_1,v_2,v_3,v_4\}$. Without loss of generality, let $v_5$ be such a vertex that is adjacent to $v_1$. Since $\delta(G)\geq 2$, $v_5$ has at at least one neighbor other than $v_1$. If $v_5$ has a neighbor in $\{v_2,v_4\}$, then there is a $5$ cycle, a contradiction with the longest cycle is a $4$-cycle. If $v_5$ has a neighbor in $V(G)\setminus \{v_1,v_2,v_3,v_4\}$, let $v_5$ be adjacent to $v_6\in V(G)\setminus \{v_1,v_2,v_3,v_4\}$. Then there is a $3$-matching $\{\{v_1v_2\},\{v_3v_4\},\{v_5v_6\}\}$ in $G$, a contradiction. Hence $N_{G}(\{v_5\})\subseteq \{v_1,v_3\}$. As $d(\{v_5\})\geq 2$, we have $N_{G}(\{v_5\})= \{v_1,v_3\}$. As $G$ is connected, there must be some vertices in $\{v_1,v_2,v_3,v_4,v_5\}$ sending edges into $V(G)\setminus \{v_1,v_2,v_3,v_4,v_5\}$. If there are some vertices in $\{v_2,v_4,v_5\}$ sending edges into $V(G)\setminus \{v_1,v_2,v_3,v_4,v_5\}$, without loss of generality, let the vertex $v_2$ be adjacent to $v_6\in V(G)\setminus \{v_1,v_2,v_3,v_4,v_5\}$. Then we find a $3$-matching $\{\{v_1v_4\},\{v_3v_5\},\{v_2v_6\}\}$ in $G$, a contradiction. Hence  $v_1$ or $v_3$ must send edges into $V(G)\setminus \{v_1,v_2,v_3,v_4,v_5\}$. 
	\par 
	Let $M$ be the vertex set satisfying that every vertex in $M$ is adjacent to $v_1$ or $v_3$. We claim that for every vertex $a\in M$ we have $N_G(\{a\})\subseteq \{v_1,v_3\}$. Otherwise, if $a$ is adjacent to $v_1$ and there is a vertex $b\in V(G)\setminus \{v_1,v_2,v_3,v_4,v_5,a\}$ such that $\{ab\}$ is an edge in $G$, then we will find a $3$-matching $\{\{ab\},\{v_1v_2\},\{v_3v_4\}\}$ in $G$, a contradiction. As $\delta(G)\geq 2$, we have $N_G(\{a\})= \{v_1,v_3\}$.
	Actually, we have $M=V(G)\setminus \{v_1,v_2,v_3,v_4,v_5\}$. Otherwise there is a contradiction with the fact that $G$ is connected. 
	\par 
	Now consider the adjacency of $v_2$ and $v_4$. If $\{v_2v_4\}$ is an edge in $G$, then we will find a $5$-cycle with the edge set $\{\{v_2v_1\},\{v_1v_5\},\{v_5v_3\},\{v_3v_4\},\{v_4v_2\}\}$ in $G$, a contradiction with the fact that the longest cycle is a $4$-cycle in $G$. Hence $v_2$ and $v_4$ are not adjacent. Considering the 
	adjacency of $v_1$ and $v_3$, if $\{v_1v_3\}$ is an edge in $G$, then $G$ is the book graph $B_{n-2}$.  If $\{v_1v_3\}$ is not an edge in $G$, then $G$ is the graph $B_{n-2}^-$.
	\par
	In conclusion, if $G$ has no $3$-matching and  $\delta(G)\geq 2$, then $G$ is the book graph $B_{n-2}$ or the graph $B_{n-2}^-$.
\end{proof}
Now we prove Theorem \ref{s32}.
\begin{proof}[Proof of Theorem \ref{s32}]
	Suppose to the contrary that there is an $n$-vertex $3$-graph $H$ with $\delta_2(H)\geq 2$ and a vertex $u$ that is not covered by a $S_3$. Let $G_u$ be the link graph of $u$. As $\delta_2(H)\geq 2$ and there is no $S_3$ covering $u$, then $\delta(G_u)\geq 2$ and there is no $3$-matching in $G_u$. By Theorem \ref{sgBt}, we have $G_u$ is the book graph $B_{n-3}$ or the graph $B_{n-3}^-$. Let $A$ be the set of vertices with degree $2$ in $V(B_{n-3}^-)$. And let $b_1$ and $b_2$ be the remained two vertices in $V(B_{n-3}^-)\setminus A$.
\par 
As $\delta_2(H)\geq 2$, we have $d_{H}(\{b_1,b_2\})\geq 2$ and $\{b_1,b_2\}$ has at least two co-neighbors. Even if $\{b_1b_2u\}$ is an edge in $H$, there still must be at least one co-neighbor of $\{b_1b_2\}$ in $A$. Without loss of generality, let $a_0\in A$ and $\{b_1b_2a_0\}$ be an edge in $H$. Consider the vertex set $A\setminus \{a_0\}$. Let $a_1,a_2$ and $a_3$ be three different vertices in $A\setminus \{a_0\}$. If $\{a_1a_2\}$ has a co-neighbor $b_1$ or $b_2$, then we will find a $S_3$ with the edge set $\{\{b_1b_2a_0\},\{b_1a_1a_2\},\{b_1ua_3\}\}$ or $\{\{b_1b_2a_0\},\{b_2a_1a_2\},\{b_2ua_3\}\}$ covering $u$, a contradiction. Hence $N_{H}(\{a_1,a_2\})\subseteq A$. Consider the induced graph $H[A\setminus \{a_0\}]$, as $\delta_2(H)\geq 2$ we have $\delta_2(H[A\setminus \{a_0\}])\geq 1$. By Theorem \ref{s22}, $H[A\setminus \{a_0\}]$ has a $P_2$ covering. Then there must be a $P_2$ in $H[A\setminus \{a_0\}]$. Without loss of generality, let $P_2$ with the vertex set $\{a_1,a_2,a_3,a_3,a_5\}$ and the edge set $\{\{a_1a_2a_3\},\{a_1a_4a_5\}\}$ be a linear $2$-path in $H[A\setminus \{a_0\}]$. As a result, we find a $S_3$ with the edge set $\{\{a_1a_2a_3\},\{a_1a_4a_5\},\{a_1b_1u\}\}$ covering $u$, a contradiction.
\par
In conclusion, if $H$ is an $n$-vertex 3-graph  satisfying that $n\geq 7$ and $\delta_2(H)\geq 2 $, then for any vertex $u\in V(H)$ there is a $3$-star $S_3$ covering $u$.
\end{proof}

\subsubsection{The proof of Theorem \ref{S3 center cover}}
In the proof of Theorem \ref{S3 center cover}, we need a useful theorem in Graph Theory as follows.

\begin{thm}\label{girth}
	Let $G$ be a simple graph and $\delta$ be the minimum degree of $G$. Then $G$ contains a cycle of length at least $\delta+1$.
\end{thm}

\begin{proof}[Proof of Theorem \ref{S3 center cover}]
Suppose to the contrary that there is an $n$-vertex $3$-graph $H$ with $\delta_2(H)\geq 3$ and a vertex $u\in V(H)$ such that there is no $S_3$ with center $u$ covering it. Let $H_u$ be the link graph of $u$. Then we have $\delta (H_u)\geq 3$ and there is no $3$-matching in $H_u$.
\par 
We claim that $H_u$ must be a connected graph. Otherwise, if $H_u$ has more than two components, then as $\delta (H_u)\geq 3$ we have there is no isolated vertices. Then selecting one edge in every component  generates a $3$-matching in $H_u$, a contradiction. If $G_u$ has two components, then we claim there must be at least one component containing a $2$-matching. Otherwise, as $\delta (H_u)\geq 3$ we have the two components can not be $3$-cycles. As the two components have no $2$-matching, the two components must be two stars, a contradiction with $\delta (H_u)\geq 3$. Hence selecting a $2$-matching in such a component and an edge in another component will generate a $3$-matching in $H_u$, a contradiction. Therefore, $H_u$ must be a connected graph.   
\par
What's more, we have the following claim.
\par
\begin{claim}
	The longest cycle in $H_u$ is a $4$-cycle.
\end{claim}
\begin{proof}
	As $\delta (H_u)\geq 3$, there is a $4$-cycle in $H_u$ by Theorem \ref{girth}. We only need to prove there is no cycle with length more than $4$ in $H_u$.\par 
	 Firstly, there is no cycle with length more than $5$. Otherwise, such cycle will generate a $3$-matching in $H_u$, a contradiction. Secondly, there is no $5$-cycle in $H_u$. Suppose to the contrary that there is a $5$-cycle in $H_u$ with the vertex set $\{v_1,v_2,v_3,v_4,v_5\}$ and the edge set $\{\{v_1v_2\},\{v_2v_3\},\{v_3v_4\},\{v_4v_5\},\{v_5v_1\}\}$. As $H_u$ is connected, then there must be a vertex in $\{v_1,v_2,v_3,v_4,v_5\}$ send one edge to $V(H_u)\setminus \{v_1,v_2,v_3,v_4,v_5\}$. Without loss of generality, let $v_1$ be adjacent to $v_6\in V(H_u)\setminus \{v_1,v_2,v_3,v_4,v_5\}$. As a result, there is a $3$-matching with the edge set $\{\{v_1v_6\},\{v_2v_3\},\{v_4v_5\}\}$ in $H_u$, a contradiction. \par 
	 Therefore, the longest cycle in $H_u$ is a $4$-cycle.
	
\end{proof}
Let $C_4$ be a $4$-cycle in $H_u$ with the vertex set $\{v_1,v_2,v_3,v_4\}$ and the edge set $\{\{v_1v_2\},\{v_2v_3\},\{v_3v_4\},\{v_4v_1\}\}$.  As $H_u$ is connected, then there must be a vertex in $\{v_1,v_2,v_3,v_4\}$ sending one edge to $V(H_u)\setminus \{v_1,v_2,v_3,v_4\}$. Without loss of generality, let $v_1$ be adjacent to $v_5\in V(H_u)\setminus \{v_1,v_2,v_3,v_4\}$. We find $v_5$ can not have a neighbor in $V(H_u)\setminus \{v_1,v_2,v_3,v_4\}$. Otherwise, let $v_6$ be a neighbor in $V(H_u)\setminus \{v_1,v_2,v_3,v_4\}$. Then we will find a $3$-matching with the edge set $\{\{v_1v_2\},\{v_3v_4\},\{v_5v_6\}\}$ in $H_u$, a contradiction. As $\delta (H_u)\geq 3$, we have $N_{H_u}(v_5)\subset \{v_1,v_2,v_3,v_4\}$.
\par 
If $\delta (H_u)= 4$, then we have $N_{H_u}(v_5)= \{v_1,v_2,v_3,v_4\}$. As $H_u$ is connected, there must be a vertex in $\{v_1,v_2,v_3,v_4,v_5\}$ sending edges to $V(H_u)\setminus \{v_1,v_2,v_3,v_4,v_5\}$. Without loss of generality, let $v_1$ be adjacent to $v_6\in V(H_u)\setminus \{v_1,v_2,v_3,v_4,v_5\}$. Then we find there is a $3$-matching with the edge set $\{\{v_1v_6\},\{v_2v_3\},\{v_4v_5\}\}$ in $H_u$, a contradiction.
\par 
If $\delta (H_u)= 3$, then we have $v_5$ has two neighbors in $\{v_2,v_3,v_4\}$. We first consider $N_{H_u}(v_5)= \{v_1,v_2,v_3\}$. As $H_u$ is connected, there must be a vertex in $\{v_1,v_2,v_3,v_4\}$ sending edges to $V(H_u)\setminus \{v_1,v_2,v_3,v_4,v_5\}$. If $v_1$ is adjacent to $v_6\in V(H_u)\setminus \{v_1,v_2,v_3,v_4,v_5\}$, then we will find a $3$-matching with the edge set $\{\{v_1v_6\},\{v_2v_5\},\{v_3\\v_4\}\}$ in $H_u$, a contradiction. If $v_2$ is adjacent to $v_7\in V(H_u)\setminus \{v_1,v_2,v_3,v_4,v_5\}$, then we will find a $3$-matching with the edge set $\{\{v_1v_5\},\{v_2v_7\},\{v_3v_4\}\}$ in $H_u$, a contradiction. If $v_3$ is adjacent to $v_8\in V(H_u)\setminus \{v_1,v_2,v_3,v_4,v_5\}$, then we will find a $3$-matching with the edge set $\{\{v_1v_4\},\{v_2v_5\},\{v_3v_8\}\}$ in $H_u$, a contradiction. If $v_4$ is adjacent to $v_9\in V(H_u)\setminus \{v_1,v_2,v_3,v_4,v_5\}$, then we will find a $3$-matching with the edge set $\{\{v_1v_5\},\{v_2v_3\},\{v_4v_9\}\}$ in $H_u$, a contradiction. Next we consider the case $N_{H_u}(v_5)= \{v_1,v_2,v_4\}$ and $N_{H_u}(v_5)= \{v_1,v_3,v_4\}$. By symmetry, we only need to consider $N_{H_u}(v_5)= \{v_1,v_2,v_4\}$. As $H_u$ is connected, there must be a vertex in $\{v_1,v_2,v_3,v_4\}$ sending one edge to $V(H_u)\setminus \{v_1,v_2,v_3,v_4,v_5\}$. If $v_1$ is adjacent to $v_6\in V(H_u)\setminus \{v_1,v_2,v_3,v_4,v_5\}$, then we will find a $3$-matching with the edge set $\{\{v_1v_6\},\{v_2v_5\},\{v_3v_4\}\}$ in $H_u$, a contradiction. If $v_2$ is adjacent to $v_7\in V(H_u)\setminus \{v_1,v_2,v_3,v_4,v_5\}$, then we will find a $3$-matching with the edge set $\{\{v_1v_5\},\{v_2v_7\},\{v_3v_4\}\}$ in $H_u$, a contradiction. If $v_3$ is adjacent to $v_8\in V(H_u)\setminus \{v_1,v_2,v_3,v_4,v_5\}$, then we will find a $3$-matching with the edge set $\{\{v_1v_4\},\{v_2v_5\},\{v_3\\v_8\}\}$ in $H_u$, a contradiction. If $v_4$ is adjacent to $v_9\in V(H_u)\setminus \{v_1,v_2,v_3,v_4,v_5\}$, then we will find a $3$-matching with the edge set $\{\{v_1v_5\},\{v_2v_3\},\{v_4v_9\}\}$ in $H_u$, a contradiction.
\par 
In conclusion, if $H$ is an $n$-vertex $3$-graph with $n\geq 7$ and $\delta_2(H)\geq 3$, then for any vertex $u\in V(H)$ we can find a $S_3$ with the center $u$.  

What's more, by observation \ref{S3center lower bound}, we have the bound in Theorem \ref{S3 center cover} is sharp.

\end{proof} 
\subsubsection{The proof of Proposition \ref{sk2}}

First we introduce a result about $k$-matching in extremal graph theory due to Erd\H{o}s and Gallai \cite{PTP} as follows.
\begin{thm}\cite{PTP} \label{kkkk}
	Let $G$ be an $n$-vertex graph. If $G$ has no $k$-matching, then $e(G)\leq \max \{\binom{2k-1}{2},\binom{n}{2}-\binom{n-k+1}{2}\}.$ 
\end{thm}
Let $k$ be an integer with $k\geq 3$. Then we prove Proposition \ref{sk2}.
\begin{proof}[Proof of $(i)$ in Proposition \ref{sk2}]
Let $H$ be an $n$-vertex 3-graph  satisfying that $n\geq 2k+1$ and $\delta_2(H)> \max \{\frac{4k^2-6k+2}{n-1},k-2-\frac{k^2-nk}{n-1}\} $. 
Let $u$ be a vertex in $V(H)$ and $G_u$ be the link graph of $u$. As $\delta_2(H)> \max\{\frac{4k^2-6k+2}{n-1},k-2-\frac{k^2-nk}{n-1}\}$, we have $\delta(G_u)> \max\{\frac{4k^2-6k+2}{n-1},k-2-\frac{k^2-nk}{n-1}\}$. By handshaking theorem, we have:
\begin{align*}
	e(G_u)&=\frac{1}{2}\sum\limits_{x\in V(G_u)} d_{G_u}(\{x\})\\
	&\geq \frac{1}{2} (n-1)\cdot \delta(G_u)\\
	&> \max \{\binom{2k-1}{2},\binom{n}{2}-\binom{n-k+1}{2}\}
	\end{align*}
 \par 
	Then by Theorem \ref{kkkk}, we have there is a $k$-matching in $G_u$, which means there is a $k$-star $S_k$ with the center $u$ in $H$. 
	 \par
	Hence, if $H$ is an $n$-vertex 3-graph  satisfying that $n\geq 2k+1$ and $\delta_2(H)> \max\{\frac{4k^2-6k+2}{n-1},k-2-\frac{k^2-nk}{n-1}\} $, then for any vertex $u\in V(H)$ there is a $3$-star $S_3$ covering $u$, where the center of $S_3$ is $u$. 
   As a direct corollary, we have  $c_2(n,S_k)\leq  \max \{\frac{4k^2-6k+2}{n-1},k-2-\frac{k^2-nk}{n-1}\}$, which ends our proof. 
   
\end{proof}	

\begin{proof}[Proof of $(ii)$ in Proposition \ref{sk2} ]
Let $H$ be an $n$-vertex 3-graph  satisfying that $n\geq 2k+1$ and $\delta_1(H)> \max \{\binom{2k-1}{2},\binom{n-1}{2}-\binom{n-k}{2}\}$. 
Let $u$ be a vertex in $V(H)$ and $G_u$ be the link graph of $u$. As $\delta_1(H)> \max \{\binom{2k-1}{2},\binom{n-1}{2}-\binom{n-k}{2}\}$, we have $e(H_u)> \max \{\binom{2k-1}{2},\binom{n-1}{2}-\binom{n-k}{2}\}$. By Theorem \ref{kkkk}, we have there is a $k$-matching in $H_u$, which means there is a $S_k$ covering $u$.
As a direct corollary, we have   $c_1(n,S_k)\leq \max \{\binom{2k-1}{2},\binom{n-1}{2}-\binom{n-k}{2}\}$, which ends our proof. 
\end{proof}

	\subsection{Path covering}
	\subsubsection{Construction}
	\begin{Constrution}
		Let $V_7$ be a vertex set. Fix $u\in V_7$, let $V'=V_7\setminus\{u\}$ and $E_7=\{u\}\vee \binom{V'}{2}$. Let $G_7=(V_7,E_7)$ be a 3-graph. 
	\end{Constrution}
	\begin{ob}\label{p2lower bound}
		$G_7$ is a $3$-graph with $\delta_2(G_7)=1$ and $\delta_1(G_7)=n-2$. There is no $P_3$ covering $u$.
	\end{ob}
	\begin{proof}
		Considering any two vertices $v_1,v_2\in V_7$, we have:
		\begin{itemize}
			\item If $v_1=u$ and $v_2\in V'$, then $d_{G_7}(\{u,v_2\})\geq 1$.
			\item If $v_1\in V'$ and $v_2\in V'$, then $d_{G_7}(\{v_1,v_2\})= 1$.
			
		\end{itemize}
		\par 
		Considering any vertex $v_0\in V_7$, we have:
		\begin{itemize}
			\item If $v_0=u$, then $d_{G_7}(\{u\})= \binom{n-1}{2}$.
			\item If $v_0\in V'$ , then $d_{G_7}(\{v_0\})= n-2$.
			
		\end{itemize}
		Hence  $\delta_2(G_7)=1$ and $\delta_1(G_7)=n-2$. Meanwhile, as all edges in $G_7$ intersect in $u$, it follows that there is no linear $3$-path $P_3$ covering $u$.
		
	\end{proof}
		\begin{Constrution}
		For $k\geq 4$, let $V_8$ be a vertex set with size more than $2k+1$. Let $A$ be a $(k-2)$-subset of $V_8$ and $B$ be the remain vertex set. Let $G_8=(A,B)$ be the complete bipartite $3$-graph.  
	\end{Constrution}
\begin{ob} \label{lower pk2}
	For $k\geq 4$, we have $\delta_2(G_8)=k-3$ and $G_8$ has no $P_k$ covering.
\end{ob}
\begin{proof}
	Considering any two vertices $v_1,v_2\in V_8$, we have:
	\begin{itemize}
		\item If $v_1\in A$ and $v_2\in B$, then $d_{G_8}(\{v_1,v_2\})= n-2$. As $n\geq 2k+1$, we have $d_{G_8}(\{v_1,v_2\})\geq  2k-1$.
		\item If $v_1$ and $v_2$  are vertices in $
	    A$, then $d_{G_8}(\{v_1,v_2\})= n-k+2$. As $n\geq 2k+1$, we have $d_{G_8}(\{v_1,v_2\})\geq  k+3$.
		\item If $v_1$ and $v_2$  are vertices in $
		B$, then $d_{G_8}(\{v_1,v_2\})= k-3$. 
	\end{itemize}
	\par 
	Hence we have $\delta_2(G_8)=k-3$. Next we prove that $G_8$ has no $P_k$-covering.
	\par 
	Let $P_l$ the longest linear path in $G_8$. If $l=k$, then let the vertex set of $P_l$ be $\{v_0,v_1,v_2,...,v_{2l-1},v_{2l}\}$ and the edge set of $P_l$ be $\{\{v_0v_1v_2\},\{v_2v_3v_4\},...,\{v_{2l-2}v_{2l-1}v_{2l}\\\}\}$. We denote the vertex set $\{v_2,v_4,...,v_{2l-2}\}$ by $A'$ and the vertex set $\{v_0,v_1,...,v_{2l-1},\\v_{2l}\}$ by $B'$. Then $P_l=(A',B')$ is a bipartite $3$-graph. As $l=k$, we have $|A|<|A'|$ and $|A|<|B'|$. Therefore, $P_l$ cannot be a subgraph of $G_8$, a contradiction with the hypothesis. Hence $l<k$ and $G_8$ has no $P_k$ covering.
\end{proof}
		\begin{Constrution}
		Let $k$ be an integer with $k\geq 3$. Let $V_9$ be an $n$-vertex set with $n\geq 4k$ and $E_9$ be a $3$-element set. Let $A$ be a $2k$-vertex subset of $V_9$ and $B$ be the remain vertex set $V\setminus A$. Let $E_9=\binom{A}{3}\cup \binom{B}{3}$. Let $G_9=(V_9,E_9)$ be a 3-graph. Actually, we have  $G_9=K_{2k}^3\cup K_{n-2k}^3.$
	\end{Constrution}
		\begin{ob} \label{k2kob}
	$G_9$ is a $3$-graph with $\delta_1(G_9)=\binom{2k-1}{2}$ and $G_9$ has no $P_k$-covering.	
	\end{ob}
	\begin{proof}
		As $|A|=2k$, there is no $P_k$ in the induced graph $G[A]$. And $A$ and $B$ are disconnected, so there is no $P_k$ covering the vertices in $A$.\par 
		Now we prove that $\delta_1(G_9)=\binom{2k-1}{2}$. Let $v$ be a vertex in $V(G_9)$.\par
		If $v\in A$, then $$d_{G_9}(\{v\})= \binom{2k-1}{2}.$$\par 
		If $v\in B$, as $n\geq 4k$ we have $$d_{G_9}(\{v\})=\binom{n-2k}{2}\geq \binom{2k-1}{2} .$$\par 
		Therefore, we have $\delta_1(G_9)=\binom{2k-1}{2}$.
		
	\end{proof}

\subsubsection{The proof of Theorem \ref{p32 exact}}
	The lower bound of $c_2(n,P_3)$ is a direct corollary of Observation \ref{p2lower bound}. Therefore, it is sufficient to show that every $3$-graph $H$ on $n$ vertices with $\delta_2(H)\geq 2$ has a $P_3$-covering.
\par 
Suppose to the contrary that there is a $3$-graph $H$ on $n$ vertices with $\delta_2(H)\geq 2$ and a vertex $u\in V(H)$ that is not contained in a copy of $P_3$. As $\delta_2(H)\geq 2$, by Theorem \ref{pp2} we have for every vertex $v_0\in V(H)$ there is a $P_2$ with the center $v_0$ covering it. Let $P_2$ be such a linear $2$-path in $H$ with the vertex set $\{u,v_1,v_2,v_3,v_4\}$ and the edge set $\{\{uv_1v_2\},\{uv_3v_4\}\}$. We denote $V(H)\setminus \{u,v_1,v_2,v_3,v_4\}$ by $A$. Let $v_5$ and $v_6$ be any two vertices in $A$. Then any vertex in $\{v_1,v_2,v_3,v_4\}$ can not be the co-neighbor of $\{v_5,v_6\}$. Otherwise, without loss of generality let $v_4$ be a co-neighbor of $\{v_5,v_6\}$. We find there is a linear $3$-path $P_3$ with the edge set $\{\{v_1v_2u\},\{uv_3v_4\},\{v_4v_5v_6\}\}$ covering $u$, a contradiction. Hence we have $N_H(\{v_5v_6\})\subseteq A\cup \{u\}.$ As $\delta_2(H)\geq 2$, there is at least one co-neighbor in $A$. Let $v_7\in A$ be a co-neighbor of $\{v_5v_6\}$.
\par
Consider the co-neighbors of $\{v_4v_5\}$. As $v_4$ is not the co-neighbor of $\{v_5,v_6\}$, we have  $N_H(\{v_4v_5\})\subseteq \{v_1,v_2,v_3,u\}.$ If $v_1$ or $v_2$ is a co-neighbor of $\{v_4v_5\}$, then we assume without loss of generality that $v_1$ is a co-neighbor of $\{v_4v_5\}$. Then we  find a linear $3$-path $P_3$ with the edge set $\{\{uv_1v_2\},\{v_1v_4v_5\},\{v_5v_6v_7\}\}$ covering $u$, a contradiction. If neither $v_1$ nor $v_2$ is a co-neighbor of $\{v_4v_5\}$, then we have $N_H(\{v_4v_5\})\subseteq \{v_3,u\}$. By Theorem \ref{pp2} we have $\{v_1,v_2,v_4\}$ must be an edge in $H$. Otherwise, there is no $P_2$ with the center $u$. Hence we  find a linear $3$-path $P_3$ with the edge set $\{\{v_1v_2v_4\},\{v_4uv_5\},\{v_5v_6v_7\}\}$ covering $u$, a contradiction.
\par 
Therefore, we have $c_2(n,P_3)=1$ for $n\geq 8$.

\subsubsection{The proof of Theorem \ref{p31}}
We first prove a lemma before we prove Theorem \ref{p31}.
\begin{lem} \label{c1 p2 center}
		If $G$ is an $n$-vertex $3$-graph satisfying that $n\geq 5$ and $\delta_1(G)\geq n-1$, then for any vertex $u\in V(G)$, we can find $4$ vertices $p,q,s,t$ where $\{u,p,q\}$ and $\{u,s,t\}$ form a linear $2$-path $P_2$ covering $u$. 
\end{lem}
\begin{proof}[The proof of the Lemma \ref{c1 p2 center}]
Let $u$ be any vertex in $V(G)$. Let $G_u$ be the link graph of $u$. As $\delta_1(G)\geq n-1$, we have: $$e(G_x)\geq n-1>\max \{\binom{4-1}{2},\binom{n-1}{2}-\binom{n-2}{2}\}.$$
By Theorem \ref{kkkk}, there is a $2$ matching in $G_u$. We assume without loss of generality that $\{p,q\}$ and $\{s,t\}$ form a $2$-matching in $G_u$. Then we find a $P_2$ with the edge set $\{\{u,p,q\},\{u,s,t\}\}$ covering $u$, which ends our proof.

\end{proof}

Now we begin to prove Theorem \ref{p31}.

\begin{proof}[The proof of Theorem \ref{p31}]

The lower bound of $c_1(n,P_3)$ is a direct corollary of Observation \ref{p2lower bound}. Therefore, it is sufficient to show that every $3$-graph $H$ on $n$ vertices with $n\geq 9$ and $\delta_1(H)\geq n+5$ has a $P_3$-covering.
\par 
Suppose to the contrary that there is a $3$-graph $H$ on $n$ vertices with $n\geq 9$ and $\delta_1(H)\geq n+5$ and a vertex $u\in V(H)$ that is not contained in a copy of $P_3$. As $\delta_1(H)\geq n-1$, by Lemma \ref{c1 p2 center} we have there is a $P_2$ with the center $u$ covering it. Let $P_2$ be  such a linear $2$-path in $H$ with the vertex set $\{u,v_1,v_2,v_3,v_4\}$ and the edge set $\{\{uv_1v_2\},\{uv_3v_4\}\}$. We denote the vertex set $V(H)\setminus \{u,v_1,v_2,v_3,v_4\}$ by $A$. Then any two vertices in $A$ has no co-neighbor in $\{v_1,v_2,v_3,v_4\}$. Otherwise, without loss of generality we assume $v_1$ is a co-neighbor of $\{v_5v_6\}$ with $v_5,v_6\in A$. Then there is a $P_3$ with the edge set $\{\{v_5v_6v_1\},\{v_1v_2u\},\{uv_3v_4\}\}$ covering $u$, a contradiction.
\par 
If there is a vertex $v$ in $\{v_1,v_2,v_3,v_4\}$ such that $\{uv\}$ has a co-neighbor in $A$, then we assume without loss of generality that $v_5\in A$ is the co-neighbor of $\{uv_1\}$. Let $H_{v_5}$ be the link graph of $v_5$. As $\delta_1(H)\geq n+5$, we have $e(H_{v_5})\geq n+5$. In $H_{v_5}$, there are at most $n-6$ edges between $A$ and $\{u,v_1,v_2,v_3,v_4\}$. And $\{u,v_1,v_2,v_3,v_4\}$ can span at most $\binom{5}{2}$ edges. Hence there is at least one edge induced in $A$. Let it be $\{v_6v_7\}$. Then we find a $P_3$ with 
the edge set $\{\{uv_1v_5\},\{uv_3v_4\},\{v_5v_6v_7\}\}$ covering $u$, a contradiction.
\par 

If there is no vertex $v$ in $\{v_1,v_2,v_3,v_4\}$ such that $\{uv\}$ has a co-neighbor in $A$, then there must exist $v_8,v_9\in A$ such that $\{uv_8v_9\}$ is an edge in $H$ as $\delta_1(H)\geq n+5>\binom{4}{2}$. Let $H_{v_1}$ be the link graph of $v_1$. 
\par 
If there is a vertex $a$ in $A$ such that $\{a,v_3\}$ or $\{a,v_4\}$ is an edge in $H_{v_1}$, then we assume without loss of generality that $\{a,v_4\}$ is an edge in $H_{v_1}$. Considering the position of $a$ in $A$, we have:
\begin{itemize}
	\item If $a$ is a different vertex form $v_8$ and $v_9$ in $A$, then we find a $P_3$ with the edge set $\{\{av_1v_4\},\{v_4v_3u\},\{uv_8v_9\}\}$ covering $u$, a contradiction.
	\item If $a$ is $v_8$ or $v_9$, then we can assume $a$ is $v_8$. Let $H_{v_8}$ be the link graph of $v_8$. As $\delta_1(H)\geq n+5$, we have $e(H_{v_8})\geq n+5$. In $H_{v_8}$, there are at most $n-6$ edges between $A$ and $\{u,v_1,v_2,v_3,v_4\}$. And $\{u,v_1,v_2,v_3,v_4\}$ can span at most $\binom{5}{2}$ edges. Hence, in $H_{v_8}$, there is at least one edge induced in $A$. Let it be $\{v_{10}v_{11}\}$. Then we find a $P_3$ with 
	the edge set $\{\{uv_1v_2\},\{uv_8v_9\},\{v_8v_{10}v_{11}\}\}$ covering $u$, a contradiction.
\end{itemize}

If there is no vertex $a$ in $A$ such that $\{a,v_3\}$ or $\{a,v_4\}$ is an edge in $H_{v_1}$, then $\{v_1v_3v_4\}$ must be an edge in $H$ as there must be at least one $P_2$ with the center $v_1$. Then we find a $P_3$ with the edge set $\{\{v_8v_9u\},\{uv_2v_1\},\{v_1v_3v_4\}\}$ covering $u$, a contradiction.
\par 
In conclusion, we have  $n-2 \leq c_2(n,P_3)\leq n+4$ for $n\geq 9$.

\end{proof}

\subsubsection{The proof of Theorem \ref{p32 position 2}}
	Let $H$ be an $n$-vertex $3$-graph with $n\geq 8$ and $\delta_2(H)\geq 3$. Let $u$ be any vertex in $V(H)$. By Theorem \ref{S3 center cover}, we have there is a $S_3$ with the center $u$ covering $u$. Let $V(S_3)$ be $\{u,v_1,v_2,v_3,v_4,v_5,v_6\}$ and $E(S_3)$ be $\{\{uv_1v_2\},\{uv_3v_4\},\{uv_5v_6\}\}$. Let $v_7$ be a vertex in $V(H)\setminus\{u,v_1,v_2,v_3,v_4,v_5,v_6\}$.
As $\delta_2(H)\geq 3$, we have $d_{H}(\{v_6v_7\})\geq 3$. Hence there is at least one co-neighbor of $\{v_6v_7\}$ in $V(H)\setminus\{u,v_5,v_6,v_7\}$. If there is a co-neighbor of $\{v_6v_7\}$ in $V(H)\setminus\{u,v_1,v_2,v_3,v_4,v_5,v_6,v_7\}$, let $v_8\in V(H)\setminus\{u,v_1,v_2,v_3,v_4,v_5,v_6,v_7\}$ be a co-neighbor of $\{v_6v_7\}$. Then we find a $P_3$ with the edge set $\{\{uv_1v_2\},\{uv_5v_6\},\{v_6v_7v_8\}\}$ covering $u$. If there is a  co-neighbor of $\{v_6v_7\}$ in $\{v_1,v_2,v_3,v_4\}$, without loss of generality let $v_1$ be a co-neighbor of $\{v_6v_7\}$. Then we find a $P_3$ with the edge set $\{\{uv_3v_4\},\{uv_5v_6\},\{v_6v_7v_1\}\}$ covering $u$.
\par 
Hence, If $H$ is an $n$-vertex $3$-graph with $n\geq 8$ and $\delta_2(H)\geq 3$, then for any vertex $u\in V(H)$ we can find a $P_3$ with the vertex set $\{u,v_1,v_2,v_3,v_4,v_5,v_6\}$ and the edge set $\{\{uv_1v_2\},\{uv_3v_4\},\{v_4v_5v_6\}\}$ covering $u$.

\subsubsection{The proof of Proposition \ref{pk2}}

%\begin{lem} \cite{FZC} \label{lema}
%	Given a $3$-graph $F$ with at least one $3$-edge, let $r$ be the maximum of $\delta_1(F')$ among all subgraphs $F'$ of $F$. Then $c_2(n,F)\leq \lfloor (1-\frac{1}{r})n+\frac{(V(F)-2r-1)}{r}\rfloor.$\end{lem}
We use the same method of Lemma 2.1 in \cite{FZC} to prove $(i)$ in Proposition \ref{pk2}.
\begin{proof}[Proof of $(i)$ in Proposition \ref{pk2}]
		The lower bound of $c_2(n,P_k)$ is a direct corollary of Observation \ref{lower pk2}. Therefore, it is sufficient to show that $c_2(n,P_k)\leq 2k-2$. We prove that if $H$ is an $n$-vertex $3$-graph with $n\geq 2k+1$ and $\delta_2(H)\geq 2k-1$, then for any vertex we can find a linear $k$-path $P_k$ covering it.\par 
For $k\geq 4$, we order the vertices of $P_k$ as $x_0,x_1,x_2,...,x_{2k-1},x_{2k}$ such that the edge set of $P_k$ is $\{\{x_0x_1x_2\},\{x_3x_4x_5\},...,\{x_{2k}x_{2k-1}x_{2k}\}\}$. 
Let $H$ be an $n$-vertex $3$-graph with $n\geq 2k+1$ and $\delta_2(H)\geq 2k-1$. Fix a vertex $v_0$ in $V(H)$. We can find a copy of $P_k$ by mapping $x_0$ to $v_0$, $x_1$ to any other vertex $v_1$ in $V(H)$, and $x_2$ to any $v_2\in N_H(\{v_0v_1\})$. Suppose that $x_0,...,x_{i-1}$ has been embedded to $v_0,...,v_{i-1}$. Considering the embedding of $x_i$ for $i\leq 2k$, if $i$ is odd, then we embedded $x_i$ to any vertex $v_i\in V(H)\setminus \{v_0,v_1,...,v_{i-1}\}$. If $i$ is even, as $\delta_2(H)\geq 2k-1$,  $\{v_{i-2}v_{i-1}\}$ has at least $2k-1$ co-neighbors. Hence $\{v_{i-2}v_{i-1}\}$ has at least one co-neighbor in $V(H)\setminus \{v_0,v_1,...,v_{i-1}\}$, let it be $v_i$. Then we embed $x_i$ to $v_i$. Continuing this process, we obtain a copy of $P_k$ when we embed the $2k+1$ vertices. 
\par 
Hence, if $H$ is an $n$-vertex $3$-graph with $n\geq 2k+1$ and $\delta_2(H)\geq 2k-1$, then for any vertex we can find a linear $k$-path $P_k$ covering it. As a direct corollary, we have 	
$c_2(n,P_k)\leq 2k-2$ for $n\geq 2k+1$ and $k\geq 4$.

\end{proof}

\begin{proof}[Proof of $(ii)$ in Proposition \ref{pk2}]
	The lower bound of $c_1(n,P_k)$ is a direct corollary of Observation \ref{p2lower bound} and Observation \ref{k2kob}. Therefore, it is sufficient to show that $c_1(n,P_k)\leq \binom{n-1}{2}-\binom{n-2k+1}{2}$. We prove that if $H$ is an $n$-vertex $3$-graph with $n\geq 4k$ and $\delta_1(H)\geq \binom{n-1}{2}-\binom{n-2k+1}{2}+1$, then for any vertex we can find a linear $k$-path $P_k$ covering it.\par 
	For any vertex $u$, suppose we have found a linear $i$-path $P_i$ with $1\leq i\leq k-1$ containing $u$. Now we prove that we can extend this $P_i$ to  $P_{i+1}$. We assume that the vertex set of $P_i$ is $\{v_1,...,v_{2i+1}\}$ and the edge set of $P_i$ is $\{\{v_1v_2v_3\},\{v_3v_4v_5\},...,\{v_{2i-1}v_{2i}\\v_{2i+1}\}\}$. As $\delta_1(H)\geq \binom{n-1}{2}-\binom{n-2k+1}{2}+1$, we have $d_H(\{v_{2i+1}\})\geq \binom{n-1}{2}-\binom{n-2k+1}{2}+1$. Hence there must be two vertices $v_{2i+2}$ and $v_{2i+3}$ in $V(H)\setminus \{v_1,...,v_{2i+1}\}$ such that $\{v_{2i+1}v_{2i+2}v_{2i+3}\}$ is an edge in $H$. Then we find a linear $(i+1)$-path $P_{i+1}$ by adding $\{v_{2i+1}v_{2i+2}v_{2i+3}\}$ to $P_i$. When $i=k-1$, we can get a linear $k$-path $P_{k}$ covering $u$.\par 
	Hence we have for $n\geq 4k$ and $k\geq 3$, $\max \{n-2, \binom{2k-1}{2}\}\leq c_1(n,P_k)\leq \binom{n-1}{2}-\binom{n-2k+1}{2}.$
	
	\end{proof}
\section{Acknowledgments}
%	\noindent {\bf Acknowledgments.}
	Ran Gu was partially supported by
	National Natural Science Foundation of China (No. 11701143).


\begin{thebibliography}{1}
		\bibitem{FZC}
		V. Falgas-Ravry and Y. Zhao, Codegree thresholds for covering 3-uniform hypergraphs, SIAM J.Discrete Math., 30(4)(2016),1899-1917.
		\bibitem{FMZT}
		V. Falgas-Ravry, K. Markstr\"{o}m and Y. Zhao, Triangle-degrees in graphs and tetrahedron coverings in 3-graphs. Combinatorics, Probability and Computing, 30(2)(2021), 175-199.
		\bibitem{YHMLE}
		L. Yu, X. Hou, Y. Ma and B. Li, Exact minimum codegree thresholds for $K_4^-$-covering and $K_5^-$-covering, The Electronic Journal of Combinatorics, 27(3)(2020), P3.22. 
		\bibitem{TMH}
		Y. Tang, Y. Ma and X. Hou, The degree and codegree threshold for linear tringle covering in 3-graphs, arxiv:2212.03718.
		\bibitem{FTDR}
		A. Freschi and A. T. Dirac-type results for tilings and coverings in ordered graphs. Forum of Mathematics, Sigma, Vol.10, Issue. 2022.
		\bibitem{FZS}
		Z. F\"{u}redi and Y. Zhao, Shadows of 3-Uniform Hypergraphs under a Minimum Degree Condition. SIAM Journal on Discrete Mathematics, 36(4)(2022), 2453-3057. 
		\bibitem{ZM}
		C. Zhang, Matching and tilings in hypergraphs, PhD thesis, Georgia State University, 2016.
		\bibitem{PTP}
		P. Erd\H{o}s and T. Gallai, On maximal paths and circuits of graphs, Acta Mathematica Hungarica, 10(3)(1959), 337-356.
	\end{thebibliography}
\end{document}